 \documentclass[12pt]{amsart}
 \usepackage{amssymb}
 \usepackage{amsmath}
\usepackage{amsfonts}

 \newtheorem{thm}{Theorem}[section]
 \newtheorem{lemma}[thm]{Lemma}
 \newtheorem{prop}[thm]{Proposition}
\newtheorem{cor}[thm]{Corollary}

\theoremstyle{definition}
\newtheorem{defn}[thm]{Definition}
\newtheorem{question}{Question}

\theoremstyle{remark}
\newtheorem{remark}[thm]{Remark}

\newtheorem*{ex}{Example}
\newtheorem{example}{Example}

\newenvironment{acknowledgement}[1][Acknowledgement]
 {\begin{trivlist} \item[\hskip \labelsep {\bfseries
#1}]}{\end{trivlist}}

\newcommand{\B}{\mathfrak{B}}

\newcommand{\Ff}{\mathbb{F}}
\newcommand{\Ss}{\mathcal{S}}
\newcommand{\R}{\mathcal{R}}

 \newcommand{\C}{\mathbb{C}}

 \newcommand{\N}{\mathbb{N}}
 \newcommand{\Zz}{\mathbb{Z}}
 \newcommand{\Q}{\mathbb{Q}}

\newcommand{\Oo}{\mathcal{O}}

 \newcommand{\al}{\alpha}
 \newcommand{\la}{\langle}
 \newcommand{\ra}{\rangle}
\newcommand{\de}{\delta}

 \newcommand{\ep}{\varepsilon}
 
 \newcommand{\LO}{\mathcal{LO}}

\begin{document}

J. Funct. Anal. to appear
 \title[semigroups]{fixed point properties for semigroups of nonlinear mappings and amenability}
 
\author[A. T.-M. Lau]{Anthony T.-M. Lau \dag}
\address{\dag \; Department of Mathematical and Statistical sciences\\
           University of Alberta\\
           Edmonton, Alberta\\
           T6G 2G1 Canada}
\email{tlau@math.ualberta.ca}
\thanks{\dag  \; Supported by NSERC Grant MS100}

\author[Y. Zhang]{ Yong Zhang \ddag}
\address{\ddag \; Department of Mathematics\\
           University of Manitoba\\
           Winnipeg, Manitoba\\
           R3T 2N2 Canada}
\email{zhangy@cc.umanitoba.ca}
\thanks{\ddag \; Supported by NSERC Grant 238949-2005}




\begin{abstract}
In this paper we study fixed point properties for  semitopological semigroup of nonexpansive mappings on a bounded closed convex subset of a Banach space. We also study a Schauder fixed point property for a semitopological semigroup of continuous mappings on a compact convex subset of a separated locally convex space. Such semigroups properly include the class of extremely left amenable semitopological semigroups, the free commutative semigroup on one generator and the bicyclic semigroup $S_1 = \la a, b:\, ab = 1\ra$. 
\end{abstract}

\maketitle

\section{introduction}\label{Intro}

Let $E$ be a Banach space and $C$ a nonempty bounded closed convex subset of $E$. The set $C$ has the fixed point property (abbreviated fpp) if every nonexpansive mapping $T$: $C\to C$ (that is $\|T(x)-T(y)\|\leq \|x-y\|$ for all $x,y\in C$) has a fixed point. The space $E$ has the fpp (respectively weak fpp) if every bounded closed (respectively weakly compact) convex set of $E$ has the fpp. The weak* fpp is defined similarly when $E$ is a dual Banach space. A result of Browder \cite{Browder} asserts that if a Banach space $E$ is uniformly convex, then $E$ has the fpp. Kirk extended this result by showing that if $C$ is a weakly compact convex subset of $E$ with normal structure (whose definition will be given below), then $C$ has the fpp. As shown by Alspach \cite{A} (see also \cite[Example~11.2]{G-K}), the Banach space $L^1[0,1]$ does not have the weak fpp (and hence not the fpp). In fact he exhibited a weakly compact convex subset $C$ of $L^1[0,1]$ and an isometry on $C$ without a fixed point.
It is well-known that $\ell^1(\Zz)$ has the weak* fpp but fails the fpp (see \cite{Lim 80}). However, in a recent remarkable paper of Lin \cite{Lin}, it was shown that $\ell^1(\Zz)$ can be renormed to have the fpp. This answered negatively a long standing question of whether every Banach space with the fpp was necessarily reflexive.

In \cite{Jap 02} Jap\'on Pineda showed that the weak fpp and the weak* fpp are equivalent for L-embedded Banach spaces which are duals of M-embedded spaces. Furthermore, Benavides, Jap\'on Pineda and Prus \cite{B-J-M} characterized weak compactness of non-empty closed convex bounded subset of an L-embedded Banach space in terms of the generic fixed point property for nonexpansive affine mappings. Also in the recent paper \cite{BGM} Bader, Gelander and Monod showed that for a nonempty bounded subset $B$ of an L-embedded Banach space $E$, there is a point in $E$ which is fixed by every linear isometry $T$ on $E$ that preserves $B$ (that is $T(B) = B$). 

Let $E$ be a separable locally convex space whose topology is determined by a family $Q$ of seminorms on $E$. We often write $(E,Q)$ to highlight the topology $Q$. A subset $C$ of $E$ is said to have ($Q$-)normal structure if, for each $Q$-bounded  subset $H$ of $C$ that contains more than one point, there is $x_0\in co H$ and $p\in Q$ such that 
\[ \sup\{p(x-x_0) :\; x\in H\} < \sup\{p(x-y) :\; x, y\in H\}, \] 
where $coH =co(H)$ denotes the convex hull of $H$. Here by $Q$-boundedness of $H$ we mean for each $p \in Q$ there is $d>0$ such that $p(x) \leq d$ for all $x\in H$. Every $Q$-compact subset has normal structure. In a uniformly convex Banach space (for example, every $L^p$ space with $p>1$) a bounded convex set always has normal structure.

Let $S$ be a semitopological semigroup, that is, a semigroup with a Hausdorff topology such that for each $t\in S$, the mapping $s\mapsto t\cdot s$ and $s\mapsto s\cdot t$ from $S$ into $S$ are continuous. The semigroup $S$ is \emph{left reversible} if any two closed right ideals of $S$ have non-void intersection, that is, $\overline{sS}\cap \overline{tS} \neq \emptyset$ for any $s,t\in S$. Here $\overline{A}$ denotes the closure of a subset A in a topological space.
Let $C$ be a subset of a locally convex topological vector space $(E,Q)$. We say that $\Ss = \{T_s:\; s\in S\}$ is a \emph{representation} of $S$ on $C$ if for each $s\in S$, $T_s$ is a mapping from $C$ into $C$ and $T_{st}(x) = T_s(T_tx)$ ($s,t\in S$, $x\in C$). Sometimes we simply use $sx$ to denote $T_s(x)$ if there is no confusion in the context. The representation is called continuous, weakly continuous or weak* continuous if each $T_s$ ($s\in S$) is $Q$-$Q$ continuous, weak-weak continuous or weak*-weak* continuous respectively. The representation is called \emph{separately} or, respectively, \emph{jointly continuous} if the mapping $(s,x)\mapsto T_s(x)$ from $S\times C$ to $C$ is separately or jointly continuous. The representation is called \emph{affine} if $C$ is convex and each $T_s$ ($s\in S$) is an affine mapping, that is,  $T_s(ax+by) = aT_sx + bT_sy$ for all constants $a,b \geq 0$ with $a +b = 1$, $s\in S$ and $x, y\in C$. We say that a representation $\Ss$ is $Q$-\emph{nonexpansive} if  $p(T_sx - T_sy) \leq p(x-y)$ for all $s\in S$, all $p \in Q$ and all $x,y \in C$. We say that $x\in C$ is a \emph{common fixed point} for (the representation of) $S$ if $T_s(x) = x$ for all $s\in S$. The set of all common fixed points for $S$ in $C$ is called the \emph{fixed point set} of $S$ (in $C$) and is denoted by $F(S)$. 
It is well-known that if $S$ is a left reversible semitopological semigroup and if $\Ss$ is a $Q$-nonexpansive representation of $S$ on $C\subset (E,Q)$,
then each of the following conditions implies that $F(S)$ is not empty:
\begin{enumerate}
\item The set $C$ is compact and convex in $(E,Q)$ (\cite{Mitch 70});
\item The set $C$ is weakly compact and convex and has normal structure (\cite{Lim 74});
\item The semigroup $S$ is discrete, the set $C$ is weakly compact and convex, and the representation $\Ss$ is weakly continuous (\cite{Hsu});
\item The set $C$ is weak* compact convex subset of $\ell^1$ (\cite{Lim 80}).
\end{enumerate}

It is also well-known that if $AP(S)$, the space of continuous almost periodic functions on $S$, has a left invariant mean, $C$ is compact, convex and $\Ss$ is a nonexpansive representation of $S$ on $C$, then $F(S)\neq \emptyset$ (see \cite{Lau73,L-Z}).

One of the purposes of this paper is to study fixed point properties for a semigroup of nonexpansive mappings on a bounded closed convex subset $C$ of a Banach space. In Section~\ref{L-embedded} we shall introduce a notion of $L$-embeddedness for subsets of a Banach space and prove that (Theorem~\ref{discrete FL}) if $S$ is left reversible, $C$ is L-embedded and $C$ contains a non-empty bounded subset $B$ such that each $T_s$ ($s\in S$) ``preserves'' $B$, that is, $T_s(B) =B$ for all $s\in S$ (which is the case when $S$ is a group and each $T_s$ maps $B$ into $B$), then $C$ contains a common fixed point for $S$. We also study in Section~\ref{Schauder} a Schauder fixed point property for a semitopological semigroup $S$, that is, every continuous representation of $S$ on a compact convex set $C$ of a separated locally convex topological vector space has a common fixed point. Such semigroups properly include the class of extremely left amenable semitopological semigroups. A semitopological semigroup $S$ is \emph{extremely left amenable} if $LUC(S)$, the space of bounded complex-valued left uniformly continuous functions on $S$, has a multiplicative left invariant mean. Extremely left amenable semigroups were studied earlier by Granirer \cite{Gran_ELA} and  Mitchell \cite{Mitch_ELA} (see also \cite{G-P}).

In Section~\ref{right zero}, we shall study the extension of an action of a semitopological semigroup $S$ on a compact Hausdorff space to $S^r$, where $S^r=S\cup \{r\}$ is obtained from $S$ by adjoining a right zero $r$ to $S$. Note that $S^r$ is in general not a semigroup since $rs$ is usually not defined for $s\in S$. This is then used to derive various fixed point properties of $S$, related to left amenability of various function spaces on $S$, on a compact convex subset of a locally convex space. Furthermore, this extension technique allows us to drop the separability condition in an earlier result of the authors in \cite{L-Z}. In Section~\ref{open questions} we list some open questions related to our work.

\section{Preliminaries}\label{Prel}
All topological spaces considered in this paper are Hausdorff. If $E$ is a locally convex space, we denote the dual space of $E$ by $E^*$.

Let $S$ be a semitopological semigroup. Let $\ell^\infty(S)$ be the $C^*$-algebra of bounded complex-valued functions on $S$ with the supremum norm and pointwise multiplication. For each $s\in S$ and $f\in \ell^\infty$, denote by $\ell_sf$ and $r_sf$ the left and right translates of $f$ by $s$ respectively, that is, $(\ell_sf)(t) = f(st)$ and $(r_sf)(t) = f(ts)$ ($t\in S$). Let $X$ be a closed subspace of $\ell^\infty(S)$ containing the constant functions and being invariant under translations. Then a linear functional $m\in X^*$ is called a mean if $\|m\| = m(1) =1$; $m$ is called a left invariant mean, denoted by LIM, if $m(\ell_sf) = m(f)$ for all $s\in S$, $f\in X$. If $X$ is a subalgebra of $\ell^\infty(S)$ then $m$ is multiplicative if $m(fg)=m(f)m(g)$ for all $f,g\in X$. Let $C_b(S)$ be the space of all bounded continuous complex-valued functions on $S$. Let $LUC(S)$ be the space of left uniformly continuous functions on $S$, that is, all $f\in C_b(S)$ such that the mappings $s\mapsto \ell_s(f)$ from $S$ into $C_b(S)$ are continuous when $C_b(S)$ has the sup norm topology. Then $LUC(S)$ is a $C^*$-subalgebra of $C_b(S)$ invariant under translations and contains the constant functions. The semigroup $S$ is called \emph{left amenable} (respectively extremely left amenable) if $LUC(S)$ has a LIM (respectively a multiplicative LIM). Left amenable semitopological semigroups include all commutative semigroups, all compact groups and all solvable groups. But the free group (or semigroup) on two generators is not left amenable. The theory concerning amenability of semigroups may be found in monographs \cite{Pat} and \cite{Pier}.

Let $AP(S)$ be the space of all $f\in C_b(S)$ such that $\LO(f) = \{\ell_sf:\, s\in S\}$ is relatively compact in the norm topology of $C_b(S)$, and let $WAP(S)$ be the space of all $f\in C_b(S)$ such that $\LO(f)$ is relatively compact in the weak topology of $C_b(S)$. Functions in $AP(S)$ (respectively $WAP(S)$) are called \emph{almost periodic} (respectively \emph{weakly almost periodic}) functions. In general, the following inclusions hold:
\[ AP(S) \subseteq LUC(S) \subseteq C_b(S) \text{ and }  AP(S) \subseteq WAP(S) \subseteq C_b(S). \]
If $S$ is discrete then
\[ AP(S) \subseteq WAP(S)\subseteq LUC(S) = \ell^\infty(S). \]
If $S$ is compact then 
\[ AP(S) = LUC(S)\subseteq WAP(S) = C_b(S). \]
If $S$ is a compact topological semigroup, that is, the multiplication is jointly continuous, then
\[ AP(S) = WAP(S)= LUC(S) = C_b(S). \]
All inclusions indicated in the above diagrams may be proper (see \cite{B-J-M} for details).

A discrete semigroup $S$ is left reversible if $sS\cap tS \neq \emptyset$ for all $s,t\in S$. If $S$ is discrete and left amenable then $S$ is left reversible. However, for general semitopological semigroup $S$, it needs not be left reversible even when $C_b(S)$ has a LIM unless $S$ is normal (see \cite{HL}).

When $S$ is a discrete semigroup the following implication relations are known.
 \begin{equation*}
\begin{matrix}
S \text{ is left amenable}\\
\Downarrow \quad \mathrel{\not \Uparrow}\\
S \text{ is left reversible}\\
\Downarrow \quad \mathrel{\not \Uparrow}\\
WAP(S) \text{ has LIM}\\
\Downarrow \quad \mathrel{\not \Uparrow}\\
AP(S) \text{ has LIM}\\
\end{matrix} 
\end{equation*}
The implication ``$S$ is left amenable $\Rightarrow$ $AP(S)$ has a LIM'' for any semitopological semigroup was established in \cite{Lau73}. During the 1984 Richmond, Virginia, conference on analysis on semigroups, Mitchell gave two examples to show that for a discrete semigroup $S$, $AP(S)$ has a LIM $\not \Rightarrow$ $S$ is left reversible (see \cite{Lau90}). The implication  ``$S$ is left reversible $\Rightarrow$ $WAP(S)$ has a LIM'' for discrete semigroups was proved by Hsu \cite{Hsu}. The facts that $WAP(S)$ has a LIM does not imply that $S$ is left reversible, and that $AP(S)$ has a LIM does not imply that $WAP(S)$ has a LIM were established by the authors recently in \cite{L-Z}.

\section{fixed point properties for L-embedded sets}\label{L-embedded}

A Banach space $E$ is L-embedded if the image of $E$ under the canonical embedding into $E^{**}$, still denoted by $E$, is an $\ell_1$ summand in $E^{**}$, that is, if there is a subspace $E_s$ of $E^{**}$ such that $E^{**} = E\oplus_1 E_s$, where $E^{**}$ is the bidual space of $E$. The Banach space $X$ is M-embedded if $X$ is an M-ideal in its bidual $X^{**}$, that is, if $X^\bot =\{\phi\in X^{***}: \phi(x) =0 \text{ for all } x\in X\}$ is an $\ell_1$-summand in $X^{***}$ (\cite{HWW}). In fact, $X$ is M-embedded if and only if $X^{***}=X^* \oplus_1 X^\bot$ (\cite[Remark~1.13]{HWW}).  The class of L-embedded Banach spaces includes all $L^1(\Sigma,\mu)$ (the space of all absolutely integrable functions on a measure space $(\Sigma,\mu)$), preduals of von Neumann algebras and the Hardy space $H_1$. Typical example of an M-embedded Banach space is $c_0(D)$ for a discrete space $D$. In fact, $c_0(D)^* = \ell^1(D)$ and $\ell^1(D)^* =\ell^\infty(D) = C(\beta D)$, where $\beta D$ is the Stone-\v Cech compactification of $D$. Let $M(\beta D)$ be the Banach space of all regular Borel measures on $\beta D$. Then 
\[ c_0(D)^{***} = C(\beta D)^* = M(\beta D) = \ell^1(D) \oplus_1 M(\beta D \setminus D). \]
It is evident that $M(\beta D \setminus D) = c_0(D)^\bot$. Thus $c_0(D)^{***} = c_0(D)^* \oplus_1c_0(D)^\bot$.
 It is also well-known that $K(H)$, the Banach space of all compact operators on a Hilbert space $H$, is M-embedded. More classical M-embedded Banach spaces may be seen in \cite[Example~1.4]{HWW}.  The dual space of an M-embedded Banach space is L-embedded. We refer to \cite{HWW} for more details of the theory concerning L-embedded and M-embedded Banach spaces. 

\begin{defn}
Let $C$ be a nonempty subset of a Banach space $E$. Denote by $\overline{C}^{\text{wk*}}$ the closure of $C$ in $E^{**}$ in the weak* topology of $E^{**}$. We call $C$ \emph{L-embedded} if there is a subspace $E_s$ of $E^{**}$ such that $E+E_s = E\oplus_1 E_s$ in $E^{**}$ and $\overline{C}^{\text{wk*}}\subset C\oplus_1 E_s$, that is, for each $u\in \overline{C}^{\text{wk*}}$ there are $c\in C$ and $\xi\in E_s$ such that $u = c+ \xi$ and $\|u\|=\|c\| + \|\xi\|$.
\end{defn}
 Trivially, every L-embedded Banach space is L-embedded as a subset of itself. From the definition, every L-embedded set $C$ in a Banach space $E$ is necessarily norm closed (even weakly closed). Indeed, we have
\[
\overline{C}^{\text{wk}}\subset \overline{C}^{\text{wk*}}\subset C\oplus_1 E_s,
\] 
where $\overline{C}^{\text{wk}}$ represents the closure of $C$  in the weak topology of $E$. If $w\in \overline{C}^{\text{wk}}$ and $w= y + w_s$ for some $y\in C$ and $w_s\in E_s$, then $w_s \in E\cap E_s$. Therefore, $w_s =0$ and $w=y\in C$. This shows that $\overline{C}^{\text{wk}} = C$, So $C$ is weakly closed. We are grateful to the referee for pointing this out.

It is readily seen that every weakly compact subset $C$ of any Banach space $E$ is L-embedded (simply take $E_s = \{0\}$ and note $\overline C^{\text{wk*}} = C$). 
It is also true that if $E$ is an L-embedded Banach space then its unit ball is L-embedded. Indeed, if $B_E$ is the unit ball of $E$ and $P$: $E^{**}\to E$ is the $\ell^1$-projection then $\overline{B_E}^{\text{wk*}}$ is the unit ball of $E^{**}$ and $P(\overline{B_E}^{\text{wk*}})=B_E$. In fact, the converse is also true. If $B_E$ is L-embedded then $\overline{B_E}^{\text{wk*}}\subset B_E\oplus_1 E_s\subset E\oplus_1 E_s$ for some subspace $E_s$ of $E^{**}$, which implies $E^{**}\subseteq E\oplus_1 E_s$ and hence $E^{**}= E\oplus_1 E_s$. Thus, a Banach space is L-embedded if and only its unit ball is L-embedded. But one cannot expect that every closed convex subset of an L-embedded Banach space is L-embedded. We will see (in the example after Theorem~\ref{FL}) that the set of all means on $\ell^\infty$ is not L-embedded as a convex closed (and even weak* compact) subset of $(\ell^\infty)^*$ although $(\ell^\infty)^*$ (as the dual space of a von Neumann algebra) is always L-embedded. For a $\sigma$-finite measure space $(\Sigma, \mu)$, it was shown in \cite[Theorem~1.1]{B-L} that a closed convex bounded set in $L^1(\Sigma, \mu)$ is L-embedded if and only if it is closed with respect to the locally in measure topology. A generalization of this result in operator algebra setting was given in \cite[Theorem~3.5]{D-D-S-T}.

As to the question whether there exist L-embedded but not weakly compact sets in a Banach space which itself is not L-embedded, we construct an example as follows to give an affirmative answer. Let $E_1$ be a Banach space which is not L-embedded and let $E_2$ be an L-embedded (non-reflexive) Banach space of infinite dimension. Suppose that $C_1$ is a weakly compact non-empty subset of $E_1$ and $C_2$ be the unit ball of $E_2$. Then it is readily seen that $C_1 + C_2$ is not weakly compact but is L-embedded in $E_1 \oplus E_2$. The latter is a Banach space which is not L-embedded.

\begin{lemma}\label{prop L-embdd}
Suppose that $C$ is a weak* 
closed subset of the dual space $X^*$ of an M-embedded Banach space $X$.
 Then $C$ is L-embedded.
\end{lemma}
\begin{proof}
Since $X^{***} = X^*\oplus_1 X^\bot$, for any $u\in\overline{C}^{\text{wk*}}$, there are $x^*\in X^*$ and $\phi\in X^\bot$ such that $u=x^*\oplus_1 \phi$. Let $(u_\al)\subset C$ be a net that converges weak* in $X^{***}$ to $u$. Then for any  $x\in X$ we have  
\[  \lim_\al u_\al(x) = u(x) = x^*(x) + \phi(x) = x^*(x). \]
So $(u_\al)$ converges weak* to $x^*$ in $X^*$. Since $C$ is weak* closed in $X^*$, $x^* \in C$.
This shows that $C$ is L-embedded.
\end{proof}

Let $C$ and $B$ be two nonempty subsets of a Banach space $E$ and $B$ is bounded. By definition the Chebyshev radius of $B$ in $C$ is
\[  r_C(B) = \inf\{r\geq 0: \exists x\in C, \;\sup_{b\in B}\|x-b\| \leq r\}. \]
Clearly we have $0\leq r_C(B) < \infty$. The Chebyshev center of $B$ in $C$ is defined to be
\[  W_C(B) =\{x\in C: \sup_{b\in B}\|x-b\| \leq r_C(B)\}.  \]
Note that, as a subset of $C$, $W_C(B)$ may be empty.

\begin{lemma}\label{Chebyshev}
Let $C$ be a nonempty L-embedded  subset of a Banach space $E$ and $B$ a nonempty bounded subset of $E$. Then the Chebyshev center $W_C(B)$ of $B$ in $C$ is nonempty and weakly compact. If $C$ is convex then so is $W_C(B)$. Moreover, if
$\Ss=\{T_s: s\in S\}$ is a norm nonexpansive representation of a semigroup $S$ on $C$ and $B\subset C$ with $B \subset T_s(B)$ for all $s\in S$, then $W_C(B)$ is also $S$-invariant.
\end{lemma}
\begin{proof} 
We first show that $W_C(B) \neq \emptyset$. From the definition of the Chebyshev radius $r_C(B)$, for each $n>0$, there is $x_n\in C$ such that
\[  \sup_{b\in B}\|x_n-b\| \leq r_C(B) + \frac{1}{n}.  \]
Let $x^{**}$ be any weak* limit point of $(x_n)$ in $E^{**}$. By L-embeddedness of $C$ there exist $c\in C$ and $\xi \in E_s$ such that $x^{**}=c+\xi$, where $E_s$ is a subspace of $E^{**}$ such that $E + E_s = E\oplus_1 E_s$. we have 
\[ \|c-b\| \leq \|c-b\| + \|\xi\| = \|x^{**} - b\| \leq \lim_{n\to \infty} r_C(B) +\frac{1}{n} = r_C(B)  \]
for all $b\in B$. This shows that $c\in W_C(B)$. 

To show that $W_C(B)$ is weakly compact we consider $\widetilde{C} = \overline{C}^{\text{wk*}}$, the weak* closure of $C$ in $E^{**}$. $\widetilde C$ is a nonempty weak* closed subset of $E^{**}$ and $\widetilde C \subset C\oplus_1 E_s$. Regard $B$ as a subset of $E^{**}$ (after canonical embedding $E$ into $E^{**}$). Using the same method one may see that the Chebyshev center $W_{\widetilde C}(B)$ of $B$ in $\widetilde C$ is nonempty. We show that $W_{\widetilde C}(B)=W_C(B)$. 
Take $x\in W_{\widetilde C}(B)$. There is $c\in C$ and $\xi \in E_s$ such that $x=c+\xi$. Then for each $b\in B$ we have $\|x-b\| = \|c-b\|+\|\xi\|$. So
\[ r_{\widetilde C}(B) \geq \sup_{b\in B}\|x-b\| = \sup_{b\in B}\|c-b\| + \|\xi\| \geq r_C(B) + \|\xi\|. \]
On the other hand, from the definition of the Chebyshev radius, 
\[ r_{\widetilde C}(B) \leq r_C(B).  \]
Therefore, it must be true that $\xi = 0$ and $r_{\widetilde C}(B)= r_C(B)$. Then it follows that $W_{\widetilde C}(B)=W_C(B)$.  

It is obvious that $W_{\widetilde C}(B)$ is weak* closed and bounded. So it is weak* compact. Since on $W_C(B)$ the weak* topology of $E^{**}$ coincides with the weak topology of $E$. We conclude that $W_C(B)$ is indeed weakly compact.

It is straightforward that, if $C$ is convex, $W_C(B)$ is convex.

Now suppose that $\Ss=\{T_s: s\in S\}$ is a norm nonexpansive representation of a semigroup $S$ on $C$ and $B\subset C$ with $B\subset T_s(B)$ for all $s\in S$. Then for each $s\in S$ and each $b\in B$, there is $b_s\in B$ such that $b=T_s(b_s)$. So
\[  \|T_s(x) - b\| = \|T_s(x) - T_s(b_s)\| \leq \|x-b_s\|\leq r_C(B) \quad (x\in W_C(B)).  \]
This leads to $T_s(x)\in W_C(B)$ for $x\in W_C(B)$. Therefore $W_C(B)$ is $S$-invariant.
\end{proof}

The hypothesis $B\subset T_s(B)$ ($s\in S$) in the above lemma is crucial for $W_C(B)$ being $S$-invariant. The condition is automatically satisfied by choose $B = T_e(C)$ if $S$ is a group and $C$ is bounded, where $e$ is the unit of $S$. The existence of such a bounded set $B$ may be also be automatically ensured in some general settings, which we are going to present.

\begin{lemma}\label{T(B)=B}
Let $S$ be a left reversible semitopological semigroup and Let $\Ss =\{T_s: s\in S\}$ be a representation of $S$ as separately continuous self mappings on a compact Hausdorff space $(K,\tau)$. Then there is a nonempty $\tau$-compact subset $B$ of $K$ such that $B\subset T_s(B)$ for all $s\in S$. If in addition that the representation is jointly continuous, then the set $B$ also satisfy $T_s(B) =B$ for all $s\in S$.
\end{lemma}

\begin{proof}
Fix an $a\in K$. Denote $W = \overline{Sa}^{\tau} = \tau\text{-cl}\{T_s(a): s\in S\}$. For each $s\in S$ let $W_s = T_s(W)$. 
Then $W$ is $\tau$-compact. So is each $W_s$ and $W_s = \tau\text{-cl}(T_s(Sa))$. We denote the closed right ideal $\overline{sS}$ of $S$ by $\hat s$. Since $S$ is left reversible, $\Sigma=\{\hat s: s\in S\}$ is naturally equipped with a partial order: $\hat s \geq \hat t$ if $\overline{sS} \subseteq \overline{tS}$. It is readily seen that $W_s =\overline{\hat s a}^{\tau}$, and so  $\{W_s: s\in S\} = \{\overline{\hat s a}^{\tau}: \hat s\in \Sigma\}$ is a decreasing net of closed subsets of $K$ in the sense that $W_s \subseteq W_t$ if $\hat s \geq \hat t$. This implies that $B = \bigcap_{s\in S}W_s \neq \emptyset$ and is $\tau$-compact due to the finite intersection property. 

To show $B\subset T_t(B)$ for each $t\in S$, we first show that for every finite subset $\al$ of $S$, $B\subset T_t(\bigcap_{s\in \al}W_s)$. From the left reversibility of $S$, there exists $s_0\in \cap_{s\in \al}\overline{sS}$. For this $s_0$ we have $W_{s_0} \subset \bigcap_{s\in \al}W_s$. Then 
\[ W_{ts_0}=T_t(W_{s_0}) \subset T_t(\bigcap_{s\in \al}W_s). \]
But $B \subset W_{ts_0}$. We then have $B\subset T_t(\bigcap_{s\in \al}W_s)$. So we have shown that
\[ T_t^{-1}(y) \cap (\bigcap_{s\in \al}W_s) \neq \emptyset  \]
for each $y\in B$ and each nontrivial finite set $\al \subset S$. By compactness
\[  T_t^{-1}(y)\cap B = T_t^{-1}(y) \cap (\bigcap_{s\in S}W_s) \neq \emptyset. \]
So there is $z\in B$ such that $T_t(z) = y$. This shows that $B\subset T_t(B)$ ($t\in S$). So the first assertion of the Lemma is true.

To show the second assertion we assume that the representation $\Ss$ is jointly continuous. We obviously have, for $t\in S$,
\[  T_t(B) \subset \bigcap_{s\in S} T_t(W_s) = \bigcap_{s\in S} W_{ts}.  \]
For each $\varsigma\in S$, from the left reversibility, we may take $\eta \in \hat\varsigma \cap \hat t$. It is evident that $\hat \eta \subset \hat\varsigma \cap \hat t$. Let $(t_\al)\subset S$ be a net such that $\eta = \lim_\al t t_\al$. We show $\bigcap_\al W_{tt_\al} \subset W_\eta$. Let $x\in \bigcap_\al W_{tt_\al}$. Then, for each $\al$, there is $b_\al \in W$ such that $x = T_{tt_\al} b_\al$. Passing to a subnet if necessary, we may assume $\tau$-$\lim_\al b_\al = b \in W$. By the joint continuity, $x = T_{tt_\al} b_\al \overset{\tau}{\to} T_\eta b \in W_\eta$. This implies $\bigcap_\al W_{tt_\al} \subset W_\eta$. Since $\bigcap_{s\in S} W_{ts} \subset \bigcap_\al W_{tt_\al}$ and  $W_\eta \subset W_\varsigma$, we then obtain $\bigcap_{s\in S} W_{ts} \subset  W_\varsigma$ for all $\varsigma\in S$. Thus
\[  T_t(B) \subset \bigcap_{s\in S} W_s = B.  \]
This is true for each $t\in S$. Combine this with the first assertion we have $T_t(B) =B$ for all $t\in S$.

\end{proof}

\begin{remark}
Under the hypothesis of Lemma~\ref{T(B)=B}, assuming the representation is jointly continuous, we may consider the collection 
\[ \B = \{B\subset K:\, B\neq \emptyset,\, B\text{ is $\tau$-compact and }T_s(B) = B \text{ for all }s\in S\}  \]
with the inclusion partial order. If $(B_\al)$ is a decreasing subchain of $\B$, then $H = \cap B_\al$ is a nonempty $\tau$-compact subset of $K$ (due to the finite intersection property of the subchain). Clearly $T_s(H) \subset  H$
 for all $s\in S$. Apply Lemma~\ref{T(B)=B} with $K$ replaced by $H$, we obtain a $B\in \B$ such that $B \subseteq H \subseteq B_\al$ for all $\al$.
Using Zorn's Lemma, we obtain a minimal element $B_{\min}$ of $\B$ which is a nonempty $\tau$-compact subset of $K$ such that $T_s(B_{\min}) =B_{\min}$ for all $s\in S$. For each $b\in B_{\min}$, applying  Lemma~\ref{T(B)=B} to $\overline{Sb}^{\tau}$, one is assured that there is $B\in \B$ such that  $B\subset \overline{Sb}^{\tau}$. By the minimality of $B_{\min}$ we must have $B_{\min}= \overline{Sb}^{\tau}$  for each $b\in B$. We will use this fact later in this section to obtain a fixed point property for left reversible semigroups.
\end{remark}

If $\tau$ is the weak-topology of a Banach space or is the weak*-topology of a dual Banach space then, as a consequence of the above lemma, we may have a bounded set $B$ such that $T_s(B) =B$ for all $s\in S$.

\begin{cor}\label{wk cpt B}
Suppose that $K$ is a weakly compact nonempty subset of a Banach space. Let $S$ be a left reversible semitopological semigroup that acts on $K$ as jointly weakly continuous self mappings. Then there is a nonempty weakly compact subset $B$ of $K$ such that $T_s(B) =B$ for all $s\in S$.
\end{cor}

\begin{cor}\label{wk* cpt B}
Suppose that $C$ is a weak* compact nonempty subset of the dual space $E^*$ of a Banach space $E$. Let $S$ be a left reversible semitopological semigroup and $\Ss =\{T_s: s\in S\}$ be a representation of $S$ as self mappings on $C$ such that the mapping $(s,x) \mapsto T_s(x)$: $S\times C \to C$ is jointly continuous when $C$ is endowed with the weak* topology of $E^*$. Then there is a weak* compact subset $B$ in $C$ such that $T_s(B) = B$ for all $s\in S$.
\end{cor}

If $S$ is discrete we particularly have the following corollaries.

\begin{cor}\label{wk dct B}
Suppose that $K$ is a weakly compact nonempty subset of a Banach space $E$. Let $S$ be a left reversible discrete semigroup and $\Ss =\{T_s: s\in S\}$ be a weakly continuous representation of $S$ on $C$. Then there is a weakly compact subset $B$ in $K$ such that $T_s(B) = B$ for all $s\in S$.
\end{cor}

\begin{cor}\label{dct B}
Suppose that $C$ is a weak* compact nonempty subset of the dual space $E^*$ of a Banach space $E$. Let $S$ be a left reversible discrete semigroup and $\Ss =\{T_s: s\in S\}$ be a weak* continuous representation of $S$ on $C$. Then there is a weak* compact subset $B$ in $C$ such that $T_s(B) = B$ for all $s\in S$.
\end{cor}

We point out that our Lemma~\ref{T(B)=B} was obtained in \cite[Lemma~5.1]{L-T} when $S$ is discrete or when $S$ is left amenable. Therefore, Corollaries~\ref{wk cpt B} and \ref{wk* cpt B} are valid when $S$ is a left amenable semitopological semigroup. Namely, they hold when $LUC(S)$ has a LIM. We note that left amenability of a semitopological semigroup $S$ does not imply left reversibility of it unless $S$ is discrete.

Lemma~\ref{Chebyshev} provides us a way to investigate fixed point properties of semigroups for closed convex sets of a Banach space. In light of Corollary~\ref{wk dct B}, the following theorem may be regarded as an extension of  Hsu's result \cite{Hsu} mentioned in Section~\ref{Intro}. This is the main theorem of this section.

\begin{thm}\label{discrete FL}
Let $S$ be a left reversible discrete semigroup. Then $S$ has the following fixed point property.
\begin{description}
\item[($F_L$)] Whenever $\Ss = \{T_s: s\in S\}$ is a representation of $S$ as norm nonexpansive self mappings on a nonempty L-embedded convex subset $C$ of a Banach space $E$ with each $T_s$ being weakly continuous on every weakly compact $S$-invariant convex subset of $C$, if $C$ contains a nonempty bounded subset $B$ such that $T_s(B) = B$ for all $s\in S$, then $C$ has a common fixed point for $S$.
\end{description}
\end{thm}
\begin{proof}
From Lemma~\ref{Chebyshev} the Chebyshev center $W_C(B)$ is an $S$-invariant nonempty weakly compact convex subset of $C$. From the hypothesis the representation $\Ss$ is  weakly continuous on $W_C(B)$. We then can apply Hsu's result of \cite{Hsu} (see (c) of Section~\ref{Intro}) to the representation of $S$ on $W_C(B)$. Since Hsu's result has never been published, we present the following argument for the sake of completeness.

We first consider the case when $S$ is countable. Let $K$ be a minimal nonempty weakly compact convex $S$-invariant subset of $W_C(B)$, and let $F$ be a minimal nonempty weakly compact $S$-invariant subset of $K$. The existence of such $K$ and $F$ is ensured by Zorn's Lemma (note the representation $\Ss$ restricting
 on $W_C(B)$ is weakly continuous). By Corollary~\ref{wk dct B}, 
 $F$ satisfies $T_s(F) = F$ for all $s\in S$. From \cite[Lemma~3.3]{L-Z}, $F$ is norm compact and hence has normal structure. If $F$ contains more than one point, then there is $x_0\in coF \subset K$ such that 
 \[  r_0 = \sup \{\|x_0-y\|: y\in F\} < \sup\{\|x-y\|: x,y\in F, x\neq y\}. \]
 Let $M =\{x\in K : \, \|x-y\|\leq r_0 \text{ for all } y\in F\}$. Then $M$ is nonempty norm closed convex (hence is also weakly closed) subset of $K$. By the nonexpansiveness of $T_s$ and the identity $T_s(F) = F$ for all $s\in S$, $M$ is $S$-invariant. But $F\not \subset M$. So $M \subsetneq K$. This contradicts to the minimality of $K$. Therefore $F$ must be a singleton. The single point in $F$ is indeed a common fixed point for $S$.
 
 Now let $S$ be a general left reversible discrete semigroup. We show that for each finite subset $\al \subset S$, there is a countable left reversible subsemigroup $S_\al$ of $S$ such that $\al \subset S_\al$. Let $\al =\{s_1, s_2, \cdots, s_n\}$. Let $\al_1 = \al$. Since $S$ is left reversible, there are $t_1, t_2,\cdots, t_n \in S$ such that $s_1t_1 = s_2t_2 = \cdots = s_nt_n$. Let $V_1 = \{s_1, s_2, \cdots, s_n, t_1, t_2,\cdots, t_n \}$ and let $\al_2 = V_1\cup V_1^2$. Here for any set $V\subset S$, $V^2$ denotes the set $\{st: \, s,t\in V\}$. Then $\al_2$ is still a finite subset of $S$. In general, when $\al_k$ ($k\geq 1$) is defined, $\al_k = \{s^{(k)}_1, s^{(k)}_{2}, \cdots, s^{(k)}_{m_k}\}$, then there are $t^{(k)}_1, t^{(k)}_{2}, \cdots, t^{(k)}_{m_k} \in S$ such that $s^{(k)}_1t^{(k)}_1= s^{(k)}_{2}t^{(k)}_{2} = \cdots = s^{(k)}_{m_k}t^{(k)}_{m_k}$. We let $V_{k} = \{s^{(k)}_i, t^{(k)}_i: i = 1,2, \cdots m_k\}$ and let $\al_{k+1} = V_{k}\cup V_{k}^2$. By induction, we have defined a sequence $(\al_k)$ of finite subsets of $S$ such that 
 \begin{enumerate}
 \item[(1)] $\al_k \subset \al_{k+1}$, $k=1,2,\cdots$,
 \item[(2)] $\al_k^m \subset \al_{k+m}$ ($k,m \in \N$), and
 \item[(3)] $\forall r_1, r_2 \in \al_k$ there are $\tau_1, \tau_2 \in \al_{k+1}$ such that $r_1\tau_1 = r_2\tau_2$.
 \end{enumerate}
 Define $S_\al = \bigcup_{k=1}^\infty \al_k$. Then $S_\al$ is a subsemigroup of $S$. (Note that if $s_1,s_2\in S_\al$, then $s_1,s_2 \in \al_k$ for sufficiently large $k$. So $s_1,s_2 \in \al_{k+1}\subset S_\al$ by the definition of $\al_{k+1}$.) Obviously, $S_\al$ is countable and left reversible, and $\al\subset S_\al$. From what has been proved for the countable case, $W_C(B)$ has a common fixed point for $S_\al$. So the fixed point set $F(\al)$ is a nonempty weakly compact subset of $W_C(B)$. Let $\Gamma$ be the collection of all finite subsets of $S$. Then $\{F(\al)\}_{\al\in \Gamma}$ is a family of weakly compact subsets of $W_C(B)$ having the finite intersection property. Hence $\bigcap_{\al\in \Gamma} F(\al) \neq \emptyset$. This shows that $W_C(B)$ has a common fixed point for $S$.
\end{proof}
 
Recall that a semitopological semigroup $S$ is \emph{strongly left reversible} if there is a family of countable subsemigroups $\{ S_\al :\; \al \in I\}$ such that
\begin{enumerate}
\item[(1)] $S = \cup_{\al\in I}S_\al$,
\item[(2)] $\overline{aS_\al} \cap \overline{bS_\al} \ne \emptyset$ for each $\al\in I$ and $a, b\in S_\al$,
\item[(3)] for each pair $\al_1, \al_2 \in I$, there is $\al_3\in I$ such that $S_{\al_1}\cup S_{\al_2} \subset S_{\al_3}$.
\end{enumerate}
The authors showed in \cite[Lemma~5.2]{L-Z} that a metrizable left reversible semitopological semigroup is always strongly left reversible. Using this fact and the same argument as in the proof of Theorem~\ref{discrete FL} one can prove the following theorem.

\begin{thm}\label{metrizable FL}
Let $S$ be a metrizable left reversible semitopological semigroup. Let $\Ss = \{T_s: s\in S\}$ be a norm nonexpansive representation of $S$ on a nonempty L-embedded  convex subset $C$ of a Banach space $E$ such that $C$ contains a nonempty bounded subset $B$ with $T_s(B) = B$ for all $s\in S$. If for every weakly compact $S$-invariant convex subset $M$ of $C$  the mapping $(s, x)\mapsto T_s(x)$ is jointly continuous from $S\times M$ into $M$ when $M$ is endowed with the weak topology of $E$, then $C$ has a common fixed point for $S$.
\end{thm}

 Again, using \cite[Theorem~3]{Lim 74} (see 
 (b) of Section~\ref{Intro}), by the same argument as the first paragraph in the proof of Theorem~\ref{discrete FL} one derives the following theorem.

\begin{thm}\label{normal structure}
Let $S$ be a left reversible semitopological semigroup and let $\Ss =\{T_s: s\in S\}$ be a representation of $S$ as norm nonexpansive and separately continuous self mappings on a nonempty L-embedded convex subset $C$ of a Banach space $E$. Suppose that $C$ has normal structure and contains a nonempty bounded subset $B$ such that $T_s(B) = B$ for all $s\in S$. Then $C$ has a common fixed point for $S$.
\end{thm}

Suppose that $C$ is a weakly compact convex subset of a separable locally convex topological space $(E,Q)$ and suppose that $\Ss = \{T_s:\; s\in S\}$ is a representation of the semigroup $S$ on $C$. Let $X$ be a translation invariant subspace of $\ell^\infty(S)$ containing the constant functions. Denote the set of all means on $X$ by $M(X)$. It is well-known that $M(X)$ is the weak* closure in $X^*$ of $co\{\ep_s:\, s\in S\}$, the convex hull of all evaluation functionals $\ep_s$ ($s\in S$).  If for each $x\in C$ and $\phi \in E^*$, the function $\phi_x$ defined by $\phi_x(s) = \la sx, \phi\ra$ ($s\in S$) belongs to $X$, then for each $m\in M(X)$ we can define $T_m$ on $C$ by $\la T_mx, \phi\ra = m(\phi_x)$ ($x\in C$, $\phi \in E^*$). Since $C$ is weakly compact, $T_m$ indeed maps $C$ into $C$. Moreover, if $m = \text{wk*-}\lim_\al u_\al$ with $u_\al = \sum{c_i^{(\al)}\ep_{s_i^{(\al)}}}\subset co\{\ep_s:\, s\in S\}$, then $T_m x= \text{wk-}\lim_\al \sum{c_i^{(\al)}s_i^{(\al)}x}$. If, in addition, $T_s$ is weakly continuous and affine then $T_s\circ T_m (x) = \text{wk-}\lim_\al \sum{c_i^{(\al)}ss_i^{(\al)}x} = T_{\ell_s^*m}(x)$, where $\ell_s$ is the left translate by $s$ and $\ell_s^*$ is the dual of $\ell_s$. In particular, if $m$ is a LIM on $X$, then $T_s\circ T_m (x) =T_m(x)$ for all $s\in S$.

The following lemma was proved for discrete semigroups in \cite{Lau76}. The proof also works in the general semitopological semigroup setting. However, since \cite{Lau76} is not readily available, a proof is included for the sake of completeness.

\begin{lemma}\label{WAP affine}
Let $S$ be a semitopological semigroup.
Suppose that $WAP(S)$ has a LIM $\mu$ and that $\Ss =\{T_s: s\in S\}$ is a representation of $S$ as separately continuous and equicontinuous affine mappings on a nonempty weakly compact convex subset $C$ of a separated locally convex space $(E,Q)$. Then there is a common fixed point for $S$ in $C$. 
\end{lemma}
\begin{proof} 
We first show that if $C$ is weakly compact convex in $(E,Q)$ and if $T$: $C\to C$ is $Q$-continuous and affine, then $T$ is also weakly continuous, that is, if $(x_\al)\subset C$ and $x_\al\overset{\text{wk}}{\to}x$, then $T(x_\al)\overset{\text{wk}}{\to}T(x)$. For this we only need to verify that, for any $\ep >0$ and $\phi \in E^*$, there is $\al_0$ such that $|\la T(x_\al) - T(x),\phi\ra| < \ep$ for $\al \geq \al_0$. Without loss of generality we may assume $T(x_\al)\overset{\text{wk}}{\to} y\in C$. Choose $\al_0$ such that
\[  |\la T(x_\al) - y,\phi\ra| < \ep/4 \text{ for } \al\geq \al_0. \]
By Mazur's Theorem $x\in \overline{co}^Q(x_\al: \al\geq \al_0)$. Since $T$ is $Q$-continuous, there is $x_\ep \in co(x_\al: \al\geq \al_0)$ such that $|\la T(x) - T(x_\ep),\phi\ra| < \ep/4$. Clearly $|\la T(x_\ep) - y,\phi\ra| < \ep/4$ since $T$ is affine. Therefore $|\la T(x) - y,\phi\ra| < \ep/2$, and $|\la T(x_\al) - T(x),\phi\ra| < \ep$ for $\al \geq \al_0$.

If the representation $\Ss$ on $C$ is affine and separately continuous, then, using the above we derive that the representation is also separately weakly continuous. If, in addition, the representation is $Q$-equicontinuous, we show that it must be weakly quasi-equicontinuous. In other words, if $T\in \overline{\Ss}^{\text{wk}}$, the closure of $\Ss$ in the product space $(C,\text{wk})^C$, then $T$: $(C, \text{wk}) \to  (C, \text{wk})$ is continuous. Consider the topological vector space $F=(E, Q)^C$ with product topology $\tau$. Then the weak topology of $F$ is precisely the product topology of $(E,\text{wk})^C$ (\cite[17.13]{K-N}). So $\overline{\Ss}^{\text{wk}}$ is the closure of $\Ss$ in $(F,\text{wk})$.
Let $\Phi = co(\Ss)$. From Mazur's Theorem, $\overline{\Phi}^{\text{wk}}$, the closure of $\Phi$ in $(F,\text{wk})$,
is the same as $\overline{\Phi}^\tau$, the closure of $\Phi$ in the $\tau$-topology of $F$. 
 From the hypothesis mappings in $\Phi$ are affine and $Q$-equicontinuous. This implies that $\overline{\Phi}^\tau$ consists only of $Q$-continuous affine mappings which are also weakly continuous from the first part of the proof. Thus, $\overline{\Ss}^{\text{wk}}\subset \overline{\Phi}^{\text{wk}}=\overline{\Phi}^\tau$ consists entirely of weakly continuous mappings or $\Ss$ is weakly quasi-equicontinuous. From \cite[Lemma~3.2]{L-Z}, $\phi_x \in WAP(S)$ for each $\phi\in E^*$ and each $x\in C$. So $T_\mu$ is well-defined and $T_\mu\in \overline{\Phi}^{\text{wk}}$. 
Now $r=\mu$ is a right zero of $S$ under the product $s\mu = \ell_s^*(\mu)$ since $\mu$ is a LIM on $WAP(S)$, where $\ell_s$: $WAP(S) \to WAP(S)$ is the left translate by $s$.
 Moreover, $T_s\circ T_\mu = T_{s\mu} = T_\mu$ since the representation is affine (see the paragraph after Theorem~\ref{normal structure}). It follows that 
 $F(S) = F(\mu) = T_\mu(C)\neq \emptyset$.
\end{proof}

In light of Lemma~\ref{Chebyshev} and Lemma~\ref{WAP affine} we have the following theorem for ``affine nonexpansive mappings on a closed convex set in a Banach space''.

\begin{thm}\label{WAP has LIM}
Let $C$ be a nonempty closed convex set in a Banach space and $S$ be a semitopological semigroup such that $WAP(S)$ has a LIM. Suppose that $\Ss=\{T_s: s\in S\}$ is a representation of $S$ as separately continuous nonexpansive affine mappings on $C$. If $C$ is L- embedded and there is a bounded set $B\subset C$ such that $T_s(B) = B$ for all $s\in S$, then $C$ has a common fixed point for $S$.
\end{thm}

\begin{remark}
If $S$ is a locally compact group then $WAP(S)$ always has a LIM and, for each $x\in C$, $B= Sx$ always satisfies $T_s(B) = B$ for all $s\in S$. So, if there is $x\in C$ such that $Sx$ is bounded, then every separately continuous nonexpansive affine representation of $S$ on a nonempty L-embedded convex set of a Banach space has a common fixed point for $S$. 
\end{remark}

Theorem~\ref{WAP has LIM} is related to Theorem~A in the recent paper \cite{BGM} of  Bader, Gelander and  Monod, which was used to give a short proof to the long standing derivation problem, that is, every continuous derivation on the group algebra $L^1(G)$ is inner for any locally compact group $G$. The problem was first solved by V. Losert in \cite{Losert}.

For discrete semigroup $S$ it is well-known that the left reversibility of $S$ implies that $WAP(S)$ has a left invariant mean and the converse is not true due to \cite[Theorem~4.11]{L-Z}. For general semitopological semigroups $S$,  the relation between the left reversibility of $S$ and the existence of a left invariant mean for $WAP(S)$ is still unknown. 

We recall that if $S$ acts on a Hausdorff space $X$, the action is quasi-equicontinuous if $\overline{S}^p$, the closure of $S$ in the product space $X^X$, consists entirely of continuous mappings \cite{L-Z}. Let $C$ be a closed convex subset of a Banach space $E$ and $S$ acts on $C$ as self mappings. We say that the action is hereditary weakly quasi-equicontinuous in $C$ if for each weakly compact $S$-invariant convex subset $K$ of $C$ the $S$-action on $(K, \text{wk})$ is quasi-equicontinuous, where $(K, \text{wk})$ is $K$ with the weak topology of $E$. We note that, if the $S$-action on $C$ is weakly equicontinuous, then it is hereditary quasi-equicontinuous in $C$. By \cite[Lemma~3.1]{L-Z}, if $C$ is weakly compact and the $S$-action on $C$ is weakly quasi-equicontinuous, then the $S$-action is hereditary weakly quasi-equicontinuous in $C$.

\begin{thm}\label{FL}
Let $S$ be a separable semitopological semigroup. Suppose that $WAP(S)$ has a left invariant mean. Then $S$ has the following fixed point property.
\begin{description}
\item[($F_L'$)] If $S$ acts on a nonempty convex L-embedded subset $C$ of a Banach space as norm nonexpansive and hereditary weakly quasi-equicontinuous mappings for which the mapping $s \mapsto T_s(x)$ ($s\in S$) is weakly continuous whenever $x$ belongs to any  weakly compact $S$-invariant convex subset of $C$ and if $C$ contains a nonempty bounded subset $B$ such that $T_s(B) = B$ ($s\in S$), then there is a common fixed point for $S$ in $C$.
\end{description}
\end{thm}

\begin{proof}
If $C$ is L-embedded and $B\subset C$ is bounded such that $T_s(B) =B$, by Lemma~\ref{Chebyshev}, $K = W_C(B)$ is nonempty weakly compact convex subset of $C$, that is, $S$-invariant. By \cite[Theorem~3.4]{L-Z} $W_C(B)$ has a common fixed point for $S$ if $WAP(S)$ has a LIM. So ($F_L')$ holds.
\end{proof}

We now give an example of bounded closed convex set in an L-embedded Banach space that is not itself L-embedded.

\begin{ex}
Let $G$ be a non-amenable group and let $C$ be the set of all means on $\ell^\infty(G)$. 
Then $C$ is not L-embedded although it is a weak* compact convex subset of $\ell^\infty(G)^*$ 
 which, as the dual space of the von Neumman algebra $\ell^\infty(G)$
, is indeed an L-embedded Banach space. As a consequence,  the set of all means on $\ell^\infty$ is not L-embedded.
\end{ex}
\begin{proof}
In fact, every group is left reversible. Let $\ell_s$ be the operator of left translation by $s$ on $\ell^\infty(G)$ ($s\in G$) and let $\ell_s^*$ be its dual operator. Then $\{\ell_s^*: s\in G\}$ is a representation of $G$ on $C$ as norm nonexpansive weak*-weak* continuous mappings. By Corollary~\ref{dct B} there is a nonempty weak* compact subset $B$ of $C$ such that $T_s(B) =B$ for all $s\in G$. On the other hand, as a bounded linear operator on a Banach space each $\ell_s^*$ is automatically weak-weak continuous. If $C$ were L-embedded then, due to Theorem~\ref{discrete FL}, $C$ would have a common fixed point for $G$ which would be a left invariant mean on $\ell^\infty(G)$. But $G$ is not amenable as assumed. There is no such an invariant mean. So $C$ is not L-embedded. 

In particular, the conclusion is true for $G = \Ff_2$. On the other hand, since $\Ff_2$ is countable, as a Banach space $\ell^\infty$ is isometrically isomorphic to $\ell^\infty(\Ff_2)$. So the set of all means on $\ell^\infty$ is not L-embedded.
\end{proof}

For a semitopological semigroup $S$, simply examining the representation of $S$ on the weak* compact convex subset of all means on $LUC(S)$ defined by the dual of left translations on $LUC(S)$, one sees that if the following fixed point property holds then $LUC(S)$ has a left invariant mean.

\begin{description}
\item[($F_*$)]\label{F*} Whenever $\Ss = \{T_s:\; s\in S\}$ is a representation of $S$ as norm non-expansive mappings on a non-empty weak* compact convex set $C$ of the dual space of a Banach space $E$ and the mapping $(s,x)\mapsto T_s(x)$ from $S\times C$ to $C$ is jointly continuous, where $C$ is equipped with the weak* topology of $E^*$, then there is a common fixed point for $S$ in $C$.
\end{description}

Whether the converse is true is an open problem 
\cite{Lau.survey}.


For a discrete semigroup acting on a subset of the dual of an M-embedded Banach space, we have the following theorem.

\begin{thm}\label{F*M}
Let $S$ be a discrete left reversible semigroup. Then $S$ has the following fixed point property.
\begin{description}
\item[($F_{*M}$)]\label{FM} Whenever $\Ss = \{T_s:\; s\in S\}$ is a representation of $S$ as weak* continuous, norm nonexpansive mappings on a nonempty weak* compact convex set $C$ of the dual space $E^*$ of an M-embedded Banach space $E$, $C$ contains a common fixed point for $S$.
\end{description}
\end{thm}
\begin{proof}
By Lemma~\ref{prop L-embdd}, $C$ is L-embedded. From Corollary~\ref{dct B} there is a nonempty weak* compact (and hence bounded) subset $B$ in $C$ such that $T_s(B) = B$ for all $s\in S$. In order that Theorem~\ref{discrete FL} can be applied here we show that every $T_s$ is indeed weakly continuous on each weakly compact $S$-invariant subset $K$ of $C$. Let $(x_\al)\subset K$ be such that $x_\al \overset{\text{wk}}{\to} x_0 \in K$. Then $T_s(x_\al) \overset{\text{wk*}}{\to}T_s( x_0)\subset K$ since $T_s$ is weak* continuous. Since the set $K$ is weakly compact, without loss of generality, we may assume $T_s(x_\al) \overset{\text{wk}}{\to}y_0\in K$. Then clearly, it is still true that $T_s(x_\al) \overset{\text{wk*}}{\to}y_0$. Thus $T_s(x_0) = y_0$ or $T_s(x_\al) \overset{\text{wk}}{\to}T_s( x_0)$ whenever $x_\al \overset{\text{wk}}{\to} x_0 \in K$. Therefore $T_s$ is weakly continuous when restricting to $K$. So we can finally conclude that $C$ contains a common fixed point for $S$ due to Theorem~\ref{discrete FL}.
\end{proof}

If $S$ is not discrete similar argument gives the following result.
\begin{thm}\label{F*Mj}
Let $S$ be a left reversible semitopological semigroup. Then $S$ has the following fixed point property.
\begin{description}
\item[($F_{*Mj}$)]\label{FMj} If $\Ss = \{T_s:\; s\in S\}$ is a norm nonexpansive representation of $S$ on a nonempty weak* compact convex set $C$ of the dual space $E^*$ of an M-embedded Banach space $E$ and the mapping $(s,x) \mapsto T_s(x)$ from $S\times C$ into $C$ is jointly continuous when $C$ is endowed with the weak* topology of $E^*$, then $C$ contains a common fixed point for $S$.
\end{description}
\end{thm}

\begin{remark} 
Since the Banach space $K(H)$ of all compact operators on a Hilbert space $H$ is M-embedded, It follows that both theorems \ref{F*M} and \ref{F*Mj} apply to $K(H)^* = J(H)$, the space of trace operators on $H$.
\end{remark}

\section{Schauder's fixed point property for semigroups}\label{Schauder}

A semitopological semigroup $S$ is called \emph{extremely left amenable} (abbreviated ELA) if $LUC(S)$ has a multiplicative left invariant mean.  Mitchell showed in \cite{Mitch_LUC} that a semitopological semigroup $S$ is ELA if and only if it has the following fixed point property.

\begin{description}
\item[($F_E$)] Every jointly continuous representation of $S$ on a nonempty compact Hausdorff space has a common fixed point for $S$.
\end{description}

For a discrete semigroup $S$, the Stone-\v Cech compactification $\beta S$ of $S$ is a semigroup with the multiplication $\phi \psi = \lim_m \lim_n s_m t_n$ if $\phi = \lim_ms_m$ and $\psi=\lim_nt_n$ for $(s_m),(t_n)\subset S$. The semigroup $\beta S$ is  a compact right topological semigroup, that is, for each $\phi\in \beta S$, the map $\psi \mapsto \psi\phi$ is continuous. Furthermore, $\beta S$ can be identified with the set of multiplicative means of $\ell^\infty(S)$. The discrete semigroup $S$ is ELA if and only if $\beta S$ has a right zero element \cite{Pat}. Note that if any two elements in $S$ have a common right zero (that is  for any $a,b\in S$ there is $c\in S$ such that $ac = bc =c$), then any finite subset of $S$ will have a common right zero which implies that $S$ is ELA. Granirer showed that for a discrete semigroup $S$ the converse is also true \cite{Gran_ELA}; in particular, if $S$ is right cancellative, then $S$ is ELA only when $S$ is trivial. A locally compact group $G$ is ELA then $G$ is trivial \cite{Gran-Lau}. But there are many examples of extremely left amenable topological groups (see \cite{G-P}). For instance, if $G$ is the group of unitary operators on a separable Hilbert space with the strong operator topology, then $G$ is ELA. Also a von Neumann algebra $\mathcal N$ is amenable if and only if the group of all unitary elements in $\mathcal N$ is the direct product of a compact group and an extremely left amenable group \cite{G-P}. We first give a new proof to the above Granirer's theorem. Our proof is very short compare to Granirer's original proof. The following fundamental result was proved by M. Katetov \cite{Kate}: Let $D$ be a set and $T$: $D\to D$ be a mapping with no fixed point in $D$. Then $D$ can be partitioned into a union of three disjoint subsets $D= A_0\cup A_1\cup A_2$ such that $A_i\cap f(A_i) = \emptyset$, $i=0,1,2$. Using this we may derive the following lemma (see \cite[Theorem~3.34]{HS} and \cite[Proposition~1]{Kate}). 

\begin{lemma}\label{T}
Let $D$ be a set and $T$: $D\to D$ be a mapping with no fixed point in $D$. Then the extension $\overline T$ of $T$ to $\beta D$, the Stone -\v Cech compactification of $D$, has no fixed point in $\beta D$ either.
\end{lemma}

\begin{thm}[\cite{Gran_ELA}]
A discrete semigroup $S$ is ELA if and only if any two elements have a common right zero.
\end{thm}

\begin{proof}
We only need to show the necessity. If a discrete semigroup $S$ is ELA, then it must be left reversible. We define a relation ``$\sim$'' on $S$ by 
$a\sim b$ if there is $x\in S$ such that $ax=bx$. It is trivial that $a\sim a$ for each $a\in S$, and if $a\sim b$ then $b\sim a$. If $a\sim b$ and $b\sim c$, then there are $x,y\in S$ such that $ax = bx$ and $by = cy$. By the left reversity, there are $z_1, z_2\in S$ such that $xz_1 = yz_2=:w$. Then $aw=axz_1 = bxz_1 = byz_2 = cyz_2 = cw$ and hence $a\sim c$. Thus, $\sim$ is an equivalence relation on $S$. We define a multiplication on $S/\sim$ by letting $[a][b]=[ab]$. This is well-defined and makes $S/\sim$ a right cancellative semigroup. In fact, if $a\sim b$ and $c\in S$, we can assume $ax = bx$ and $xz= cy$. So $acy = axz= bxz=bcy$. This shows $ac\sim bc$ if $a\sim b$. Trivially, we also have $bc\sim bd$ if $c\sim d$. Thus, $ac\sim bd$ if $a\sim b$ and $c\sim d$. 

If $S$ is ELA, then so is $S/\sim$, as a homomorphic image of $S$. We show that any right cancellative ELA semigroup is trivial (i.e. contains only one element if it is not empty). This will imply that $a\sim b$ for any two elements $a, b\in S$ if $S$ is ELA. So there is $y\in S$ such that $ay=by$, and there is $z\in S$ such that $(ay)z = yz $. Let $c=yz$. We then have $ac = bc = c$, that is,  $c$ is a common right zero for $a$ and $b$.

Suppose that $S$ is right cancellative and ELA. If $S$ contains more than one element, then there is $a\in S$ which is not a right identity. Otherwise every element of $\beta S$ is right identity for $S$ and $S$ cannot be ELA. Define $T$: $S\to S$ by $T(s) = as$. Then $T$ has no fixed point in $S$ (if $as=s$ for some $s\in S$, then $bas=bs$. From the right cancellative law $ba=b$ for all $b\in S$). By Lemma~\ref{T}, $\overline T$: $\beta S \to \beta S$ has no fixed point. In particular, there is no $m\in \beta S$ such that $am=m$. So $S$ cannot be ELA.
\end{proof}

Related to the fixed point property $(F_E)$ we may consider the following Schauder's fixed point property for a semitopological semigroup $S$.
\begin{description}
\item[($F_S$)]\label{Fc} Every jointly continuous representation of $S$ on a nonempty compact convex set $C$ of a separated locally convex topological vector space has a common fixed point.
\end{description}

From the definition, ($F_S$) is weaker than ($F_E$), that is,  every ELA semigroup has fpp ($F_S)$. On the other hand, many semigroups has ($F_S$) without being ELA.

\begin{example}
The well-known Schauder's Fixed Point Theorem can be stated as: the free commutative (discrete) semigroup on one generator has the fixed point property ($F_S$).
\end{example}

\begin{example} A cyclic group (finite or non-finite) has the fixed point property ($F_S$).
\end{example}

\begin{proof} Let $g$ be the generator. Then $F(g)\neq \emptyset$ and $F(g^{-1}) = F(g)$.
\end{proof}

\begin{example}\label{FS2} Commutative free semigroup on $n\geq 2$ generators does not have ($F_S$).
\end{example}

\begin{proof} It suffices to show the case $n=2$.  In this case the semigroup is isomorphic to $N_0\times N_0$, where $N_0$ is the additive semigroup of non-negative integers $\{0,1,2,\cdots\}$ (which is the free commutative semigroup on one generator). It is known that there are two continuous functions $f$ and $g$ mapping the unit interval $[0,1]$ into itself which commute under the function composition but do not have any common fixed point in $[0,1]$ \cite{boyce}.  We consider the representation of $N_0\times N_0$ on $[0,1]$ defined by $T_{(0,0)} = id_{[0,1]}$, $T_{(1,0)} = f$ and $T_{(0,1)} = g$. Then $N_0\times N_0$ has no common fixed point on $[0,1]$ although $N_0$ has the fpp ($F_S$). 
\end{proof}

\begin{example}\label{Q}
The additive semigroup $(\Q_+,+)$ has the fpp $(F_S)$, so does the additive group $(\Q,+)$ where $\Q$ is the set of all rational numbers and  $\Q_+$ represents the set of all positive rational numbers.  
\end{example}

\begin{proof} We only show the semigroup case. The group case is similar. If $\Q_+$ acts on a compact convex set $K$, then the fixed point set $F(s)$ is not empty for each $s\in \Q_+$. Clearly, if $s = nt$ for some integer $n$, then $F(t) \subset F(s)$. Let $s_1, s_2, \cdots s_k$ be finite elements of $\Q_+$. We can write $s_i=\frac{m_i}{n_i}$ ($i=1, 2, \cdots, k$), where $n_i$ and $m_i$ are integers. Let $s=\frac{1}{n_1 n_2 \cdots n_k}$. Then all $s_i$ are multiples of $s$. So $F(s)\subset F(s_i)$ for all $i$. Therefore the collection $\Gamma =\{F(s) :\, s\in \Q_+\}$ of compact subsets of $K$ has the finite intersection property. Thus $\cap_{s\in \Q_+}F(s) \neq \emptyset$, showing that $\Q_+$ has a common fixed point in $K$.
\end{proof}

\begin{example}
Recall that a semilattice is a commutative semigroup of idempotents. Any semilattice has the fpp $(F_S)$. In fact, a semilattice is extremely left amenable and hence has the fpp $(F_E)$. 
\end{example}
\begin{proof} Let $S$ be a semilattice. For any finite elements $s_1, s_2, \cdots, s_k\in S$, let $s=s_1 s_2 \cdots s_k$. Then $s_is = s$ and so $F(s)\subset F(s_i)$ for each $i$. Therefore $\{F(s):\, s\in S\}$ has finite intersection property.
\end{proof}

\begin{example}\label{bicyclic}
The bicyclic semigroup, denoted by $S_1 = \la a, b \,|\; ab = e \ra$, is a semigroup generated by a unit $e$ and two more elements $a$ and $b$ subject to the relation $ab =e$. This semigroup is not ELA but has the fpp ($F_S$). 
\end{example}

\begin{proof} In fact, elements of $S$ are of the form $s=b^na^m$ for some integers $n\geq 0$ and $m\geq 0$. If $bs=s$, then $b^{n+1}a^m = b^na^m$. Multiply $a^n$ to the left and $b^m$ to the right of the two sides of the last equation. we then have $b=e$, which is a contradiction. So $S^1$ is not ELA.

Let $\Ss$ be a continuous representation of $S_1$ on a nonempty compact convex set $K$. From Schauder's fixed point theorem, $K$ has a fixed point, say $x$, for $b$.  Then $b(ex) = (be)x = (eb)x =e(bx) = ex$, $e(ex) = e^2x = ex$ and $a(ex) = a(bex) = (ab)ex =e^2x = ex$. Therefore, $ex$ is a common fixed point in $K$ for the generators $a$, $b$ and $e$. This shows that $ex$ is a common fixed point for $S_1$.
\end{proof}

\begin{remark} The proof in fact shows that $F(S_1) = eF(b) = F(b)$. Since the fixed point property ($F_S$) implies that $S$ is left amenable by the next proposition, we obtain another proof for the left amenability of $S_1$. Other proofs can be found in \cite{D-N} and \cite{L-Z}.
\end{remark}

\begin{prop}\label{Fc LIM}
If a semitopological semigroup $S$ has the fpp ($F_S$), then $LUC(S)$ has a left invariant mean. In particular, if $S$ is a discrete semigroup with the fpp ($F_S$), then $S$ is left amenable.
\end{prop}
\begin{proof}
Let $M(S)$ be the set of all means on $LUC(S)$. Then $M(S)$ is weak* compact convex subset of $LUC(S)^*$ which is a separated locally convex space with the weak* topology. Consider the left translation operator $\ell_s$ on $LUC(S)$. Then its dual mapping $\ell_s^*$ defines a jointly (weak*) continuous representation of $S$ on $M(S)$. The common fixed point of this representation gives a left invariant mean on $LUC(S)$. 
\end{proof}

\begin{example}
Let $S$ be a semitopological semigroup. If $LUC(S)$ has a left invariant mean that is in the convex hull of $\Delta(S)$, the spectrum of $LUC(S)$, then $S$ has the fpp ($F_S$). In fact, by \cite[Theorem 2.2]{Lau-Trans 152} (see also \cite{Lau-Trans 148}), there is a finite set $F\subset C$ such that $s(F) = F$ for all $s\in S$. Let $F =\{x_1,x_2, \cdots, x_n\}$. Then the element $\frac{1}{n}\sum_{i=1}^n{x_i}\in C$ is a common fixed point in $C$.
\end{example}

Consider  $S_2 = \la e, a, b, c \;|\; ab = e, ac = e\ra$, the semigroup generated by the unit $e$ and three more elements $a, b$ and $c$ subject to the relations $ab =ac = e$. We know that $S_2$ is not left amenable. So it does not have the fpp ($F_S$). Similarly, the partially bicyclic semigroups $S_{1,1}$, generated by a unit $e$ and four more elements $a, b, c$ and $d$ subject to the relations $ab = e, cd = e$, does not have the fpp ($F_S$). It is worth to mention that $WAP(S_2)$ has a LIM and $AP(S_{1,1})$ has a LIM (\cite{L-Z}).

One can always associate a unit $e$ to a non-unital semigroup $S$ to get a unitization semigroup $S^e$ of $S$. If $S$ is a semitopological semigroup, then $S^e$ is still a semitopological semigroup whose topology extends that of $S$ by adding $\{e\}$ to the open set base of $S$.

\begin{prop} A semitopological semigroup $S$ has ($F_S$) if and only if $S^e$ does.
\end{prop}
\begin{proof} The ``if'' part is trivial since from a representation of $S$ one can always get a representation of $S^e$ by representing $e$ as the identity mapping. The converse is also clear if one notices $F(S^e) = eF(S)$.
\end{proof}

\begin{prop}\label{homomorph}
Suppose that there is a semigroup homomorphism $\gamma$: $S \overset{\rm{onto}}{\to} R$, then $R$ has ($F_S$) (resp. ($F_E$)) if $S$ does.
\end{prop}
\qed

Note that the free group on two generators is not left amenable and is a homomorphism image of $S_{1,1}$. So again we derive that $S_{1,1}$ does not have ($F_S$) from the above proposition.

From Proposition~\ref{homomorph} it is clear that if the product semigroup $S\times R$ has the fpp ($F_S$), then both term semigroups $S$ and $R$ must have ($F_S$). But the converse is not true. A counterexample is $N_0\times N_0$. As we have seen, it does not have the fpp ($F_S)$ although $N_0$ does. However, a partial converse is true.

\begin{prop} Let $S$ and $R$ be unital semitopological semigroup.
If $S$ is ELA and $R$ has ($F_S$), then $S\times R$ has ($F_S$).
\end{prop}
\begin{proof}
Let $e_S$ and $e_R$ be units of $S$ and $R$ respectively.  Let $\Ss\R =\{T_{(s,t)} :\; (s,t)\in S\times R\}$ be a jointly continuous representation of $S\times R$ on $K$, a compact convex subset of a locally convex space $E$. Then $\Ss =\{T_s = T_{(s,e_R)}: \; s\in S\}$ and $\R =\{T_t = T_{(e_S,t)}: \; t\in R\}$ are jointly continuous representations of $S$ and $R$, respectively. $R$ has ($F_S$). So $F(R) \ne \emptyset$ and is compact. Now for each $s\in S$, $sF(R) = T_sF(R) \subset F(R)$ since $T_s$ commutes with every $T_t$. Hence $F(R)$ is $S$-invariant. But $S$ is ELA and ($F_E$) holds for $S$. Therefore, there is $x\in F(R)$ Such that $T_sx = x$ for all $s\in S$. This shows that $x$ is a common fixed point for $S\times R$.
\end{proof}
 \begin{prop}\label{dense}
 If a semitopological semigroup $S$ has a dense subsemigroup that has the fpp $(F_S)$, then $S$ has the fpp $(F_S)$.
 \end{prop}
 \qed

 From this proposition and Example~\ref{Q}, we immediately have the following.
 
 \begin{example}
 The additive semigroup $(\mathbb R_+,+)$ and the additive group $(\mathbb R, +)$ both have the fpp $(F_S)$.
 \end{example}
\qed
 
 \begin{example}
 The multiplicative group $(\mathbb R_+,\cdot)$ has the fpp $(F_S)$.
 \end{example}
 
 \begin{proof}
 $(\mathbb R_+,\cdot)$ is isomorphic to $(\mathbb R,+)$ under the isomorphism $r\mapsto \ln r$.
 \end{proof}
 
 For a prime number $p$ the field of $p$-adic numbers, denoted by $\Q_p$, is the completion of the rational numbers $\Q$ with respect to $p$-adic norm defined by $|r|_p=p^{-m}$ if $r\in \Q$ is (uniquely) written as $r = p^mq$ where $m\in \mathbb Z$ and $q$ is a rational number whose numerator and denominator are not divisible by $p$. $(\Q_p, +)$ is a commutative locally compact group with dense subset $\Q$. So by Proposition~\ref{dense} we have the following.
 
 \begin{example}
 The additive locally compact topological group of $p$-adic numbers $(\mathbb Q_p,+)$ has the fpp $(F_S)$.
 \end{example}
 \qed
 
 Since the compact group of the unit circle is isomorphic to $\mathbb R/\mathbb Z$, we then have
 
 \begin{example}
The unit circle has the fpp $(F_S)$.
 \end{example}
 \qed
 
 Similarly, the compact $p$-adic group has the fpp $(F_S)$.
 
 It is known that a locally compact group with a dense subgroup isomorphic to the cyclic group is either the cyclic group or a compact group (\cite[9.2]{HR}). This group has the fpp $(F_S)$.
 
Although $S_2$ does not have the fpp ($F_S$). It has a weaker version of ($F_S$). We recall that a representation $\Ss$ of a semigroup $S$ on a closed subset $K$ of a locally convex space is quasi-equicontinuous if the closure $\overline{\Ss}$ of $\Ss$ in the product space $K^K$ contains only continuous mappings \cite{L-Z}.

Let $\beta S$ be the Stone-\v Cech compactification of $S$. Then every representation of $S$ on any compact Hausdorff space $K$ can be extended to a representation of $\beta S$ on $K$. If the $S$-representation is equicontinuous or quasi-equicontinuous, then so is the extension. We therefore have the following result.

\begin{prop}
A semigroup $S$ has either of the following fixed point property if and only if $\beta S$ does.
\begin{description}
\item[($F_{ec}$)]\label{Fec} Every equicontinuous representation of $S$ on a nonempty compact convex set $K$ of a locally convex topological space E has a common fixed point in $K$.
\end{description}
\begin{description}
\item[($F_{qec}$)]\label{Fqec} Every quasi equicontinuous representation of $S$ on a nonempty compact convex set $K$ of a locally convex topological space E has a common fixed point in $K$.
\end{description}
\end{prop}
\qed

An example of non-amenable semigroup that has the fpp $(F_{qec})$ is $S_2$.
\begin{example}
The partially bicyclic semigroup $S_2$ has the fpp $(F_{qec})$.
\end{example}
\begin{proof}
Denote by $S_2^w$ the weakly almost periodic compactification of $S_2$. With the weak* topology of $WAP(S_2)^*$, $S_2^w$ is a compact semi-topological semigroup containing $S_2$ as a dense subsemigroup. Let $\Ss_2$ be a quasi equicontinuous representation of $S_2$ on $K$. We can extend $\Ss_2$ to a representation of $S_2^w$ on $K$ as defined in the following:

For any $x\in K$, $\phi\in E^*$, define $\phi_x$: $S_2\to \C$ by $\phi_x(s) = \la sx, \phi\ra$ ($s\in S_2$). From \cite[Lemma 3.2]{L-Z}, $\phi_x\in WAP(S_2)$. For each $u \in S_2^w$ we then define $T_u$: $K \to K$ by $\la T_ux, \phi\ra = u(\phi_x)$ ($x\in K$, $\phi\in E^*$). It is readily seen that $T_{u\cdot v} = T_u\circ T_v$ ($u,v\in S_2^w$), and if $u_\al \to u$ in $S_2^w$ then $T_{u_\al}x \to T_ux$ ($x\in K$). (Here we note that since $K$ is compact, on $K$ the weak topology coincides the original topology carried from $E$.)

From \cite[Lemma 4.9]{L-Z} there is an element $e_A \in A = \cap_{n=1}^{\infty} \overline{A_n} \subset S_2^w$, where $A_n = \{a^n, a^{n+1}, \cdots\}$, such that $be_A = e_Ab$ and $ce_A = e_A c$. This implies that $be_A = ce_A$ since $e_A = \lim_ia^{m_i}$ for some net $(a^{m_i})$ of powers of $a$ and $e_Ab = e_Ac = \lim_ia^{m_i-1}$. So $T_b\circ T_{e_A} = T_c\circ T_{e_A}$.

Since the subsemigroup $\la a, b, e\ra$ is a copy of the bicyclic semigroup $S_1$, from Example~\ref{bicyclic}, $F(a)\cap F(b)\cap F(e) \neq \emptyset$. Let $x\in F(a)\cap F(b)\cap F(e)$. Then $x\in F(e_A)$ since $F(e_A)\supset F(a)$. So, $cx = T_c\circ T_{e_A}x = T_b\circ T_{e_A}x = bx =x$. Therefore, $x\in F(c)$ and hence $x$ is a common fixed point for $S_2$. In fact, we have shown that $F(S_2) = eF(b) = eF(c)$.
\end{proof}

It is not known to the authors whether $S_{1,1}$ has the fpp $(F_{ec})$. We note that $AP(S_{1,1})$ has a LIM \cite{M2,L-Z}. Therefore, according to \cite[Theorem]{Lau73}, any equicontinuous affine representation of $S_{1,1}$ on a compact convex set of a locally convex space has a common fixed point. 

\begin{prop}
Suppose that $S$ is a semigroup having the fpp ($F_S$). Then, for any compact subset $K$ of a locally convex space, if $K$ is a continuous retract of $\overline{co}K$, any continuous representation $\Ss = \{T_s:\; s\in S\}$ of $S$ on $K$ has a common fixed point in $K$.
\end{prop}
\begin{proof}
Let $\tau$: $\overline{co}K \to K$ be a retraction. Then $\Ss^\tau = \{T_s\circ\tau:\; s\in S\}$ is a continuous representation of $S$ on $\overline{co}K$. The common fixed points for $\Ss^\tau$ are common fixed points for $\Ss$ in $K$.
\end{proof}
\begin{remark}\label{fg}
Even if $S$ has common fixed points in a compact set $K$, one may not expect that $F(S)$, the fixed point set of $S$ in $K$, to be a continuous retract of $\overline{co}F(S)$. For example: Let $f$ and $g$ be the functions described in the proof of Example~\ref{FS2}. By Schauder's Fixed Point Theorem, $F(S_f)\neq \emptyset$, where $S_f$ is the commutative composition semigroup generated by $f$. We claim that $F(S_f)$ is not a retract of $\overline{co}F(S_f)$. In fact, $F(S_f)$ is a compact subset of $[0,1]$, and it is invariant under $g$ since $f$ and $g$ are commutative. If $F(S_f)$ were a retract of $\overline{co}F(S_f)$, then $g$ would have a fixed point in $F(S_f)$ (which would be a common fixed point of $f$ and $g$) due to the above proposition. This contradiction shows that our claim is true.
\end{remark}

Let $S$ be a semigroup and let $X$ be a translation invariant subspace of $\ell^\infty(S)$ containing the constant functions. Suppose that $\Ss = \{T_s:\; s\in S\}$ is a representation of $S$ as nonexpansive mappings on a weakly compact convex subset $K$ of a Banach space (or a separable locally convex topological space) $E$. If for each $x\in K$ and $\phi \in E^*$, the function $\phi_x$ defined by $\phi_x(s) = \la sx, \phi\ra$ ($s\in S$) belongs to $X$, then we can define $T_m$: $K\to K$ for each mean on $X$ by $\la T_mx, \phi\ra = m(\phi_x)$ ($x\in K$, $\phi \in E^*$). If, in addition, $X$ has a LIM $\mu$, we are interested  in whether or not $F(S) = F(T_\mu)$.
It has been shown in \cite{L-M-T} that this is the case when $X = AP(S)$ and $K$ is compact. One sees from the proof of Lemma~\ref{WAP affine} that this is also the case when $X = WAP(S)$ and the the representation is separately continuous and equicontinuous affine. Here we consider the case when $WAP(S)$ has a multiplicative LIM.

\begin{prop}
Suppose that $WAP(S)$ has a multiplicative left invariant mean $\mu$. Let $\Ss$ be a weakly quasi-equicontinuous representation of $S$ on a nonempty weakly compact convex subset of a Banach space $E$. Then $F(S)\neq \emptyset$ and $F(S) = F(\mu)$.
\end{prop}
\begin{proof}
Since the representation on $K$ is weakly quasi equicontinuous, we have $\phi_x\in WAP(S)$ for each $x\in K$ and $\phi\in E^*$ by \cite[Lemma 3.2]{L-Z}. So $T_m$ is well-defined and weakly continuous on $K$ for each mean $m$ on $WAP(S)$. If in addition that $m$ is a multiplicative mean, then we have further that $T_{\ell_s^*(m)} = T_s\circ T_m$ ($s\in S$). So for multiplicative left invariant mean $\mu$, we have $T_s\circ T_\mu = T_\mu$ ($s\in S$). This implies that $T_\mu(K) \subset F(S)$. In particular, $F(T_\mu)\subset F(S)$. On the other hand, it is clear from the definition of $T_m$ that $F(S) \subset F(T_m)$ for each mean $m$ on $WAP(S)$. Hence $F(S) = F(T_\mu)$, which is not empty from  Schauder's fixed point theorem.
\end{proof}
\begin{remark}
The proof also shows that $F(S)$ is a retract of $K$ and $T_\mu$ is a retraction from $K$ onto $F(S)$.
\end{remark}

\section{Extension to $S^r=S\cup\{r\}$}\label{right zero}

We first observe that if $S$ has a right zero element (that is  an element $r\in S$ such that $sr = r$ for all $s\in S$), then every representation of $S$ on any nonempty set $K$ has a common fixed point for $S$. In fact, $F(S) = F(r)=T_r(K) \neq \emptyset$. An example of such a semigroup is the Cuntz semigroup, which is the semigroup $\Oo_n$ ($n>1$) generated by elements $\{p_i,q_i:\, i=1,\ldots,n\}$ subject to constraints $p_iq_j=0$ for $i\neq j$ and $p_iq_i = 1$ (see \cite{Ren, Pat_groupoids} for details). In general, we may associate a right zero element $r$ to a semigroup $S$. Note that this does not result in a new semigroup because $rs$ is usually not defined for $s\in S$. We denote this extended object by $S^r = S\cup\{r\}$. In it the product $st$ is well defined for $s\in S$ and $t\in S^r$. Also $tr=r$ for all $t\in S^r$. A representation $\Ss^r = \{T_t:\; t\in S^r\}$ of $S^r$ on $K$ is understood as a collection of self mappings on $K$ satisfying $T_{st}x =(T_s\circ T_t)x = T_s(T_tx)$ ($s\in S$, $t\in S^r$ and $x\in K$) and $T_r\circ T_r = T_r$. It is still true that any representation of $S^r$ on a nonempty set $K$ has a common fixed point for $S^r$ and $F(S^r) = F(r) = T_r(K)$.

\begin{lemma}\label{S0}
Let $S$ be a semigroup, and let $\Ss = \{T_s:\; s\in S\}$ be a {representation} of $S$ on a nonempty set $K$. Then $\Ss$ has a common fixed point in $K$ if and only if the representation can be extended to a representation of $S^r$ on $K$.
\end{lemma}
\begin{proof}
If the representation can be extended to $S^r$, then, as we observed in the beginning of the section, $F(S^r)\neq \emptyset$. So $F(S)\neq \emptyset$. Conversely, if $F(S)\neq \emptyset$, let $x_0\in F(S)$ and define $T_r(x) = x_0$ ($x\in K$). It is readily seen that $\Ss^r = \Ss \cup \{T_r\}$ is a representation of $S^r$ on $K$ (in fact, $T_{t}\circ T_{r} = T_r = T_{tr}$ for $t\in S^r$).
\end{proof}

Let $S$ be a semitopological semigroup. 
Suppose that $S$ acts on a Hausdorff space $K$ and the representation is separately continuous.  Given $s\in S$ and $h\in C(K)$, we define $_sh \in C(K)$ by $_sh(x) = h(sx)$ ($x\in K$). Similarly, given $k\in K$ and $h\in C(K)$, we define $h_k \in C(S)$ by $h_k(s) = h(sk)$ ($s\in S$). If $K$ is compact, then $h_k\in C_b(S)$. 

\begin{lemma}\label{MLIM}
Let $S$ be a semitopological semigroup acting on a compact Hausdorff space $K$ as separately continuous self-mappings. Suppose that there is a left invariant 
subalgebra $X$ of $C_b(S)$ containing the constant functions and there is a subset $Y$ of $C(K)$ such that $Y$ separates points of $K$, $_sh \in Y$ and $h_k\in X$ for $h\in Y$, $s\in S$ and $k\in K$. If $X$ has a multiplicative left invariant mean, then the representation of $S$ on $K$ can be extended to a representation of $S^r$ on $K$ and so $S$ has a common fixed point in $K$.
\end{lemma}
 \begin{proof}
If $m$ is a multiplicative mean on $X$, then there is a net $(s_i)\subset S$ such that $m = \text{wk*}\lim \de_{s_i}$. Here $\de_s$ denotes the evaluation functional on $X$, that is,  $\de_s(f) = f(s)$ ($f\in X$). Let $h\in Y$. Then $m(h_k) = \lim_i h(s_ik) \in h(K)$ since $h(K)$ is compact. Denote
\[  K_{h,k} = \{x\in K :\, h(x) = m(h_k)\}, \quad (k\in K, h\in Y). \]
Clearly, $K_{h,k}$ is a compact subset of $K$. For each $k\in K$, $(K_{h,k})_{h\in Y}$ has finite intersection property. In fact, if $h^1, h^2, \cdots, h^n \in Y$ then the mapping $\Gamma$: $K\to \C^n$ defined by $\Gamma(x) = (h^1(x), h^2(x), \cdots, h^n(x))$ is continuous and $\Gamma(K)$ is compact. So,
\begin{align*} (m(h^1_k), m(h^2_k), \cdots, m(h^n_k)) & = \lim_i (h^1(s_ik), h^2(s_ik), \cdots, h^n(s_ik))\\
                                         & = \lim_i\Gamma(s_ik) \in \Gamma(K) . 
                                          \end{align*}
This shows that $\cap_{j=1}^n K_{h^j, k} \neq\emptyset$. Therefore $\cap_{h\in Y} K_{h, k} \neq\emptyset$. In fact, $\cap_{h\in Y} K_{h, k} = \{x_k\}$ is a singleton since $Y$ separates the points of $K$. So we may extend the $S$ action on $K$ to $S^r$ by defining $T_r(k) = x_k$ ($k\in K$). If $X$ has a multiplicative left invariant mean, still denoted by $m$, then this extension is a representation of $S^r$ on $K$. To see this we compute
\begin{align} 
 h(sT_r(k)) = {_sh}(T_r(k)) = m((_sh)_k) &= m(\ell_s(h_k)) \label{extension}\notag \\
                                         & = m(h_k) = h(T_r(k))
\end{align}
for $s\in S$, $k\in K$ and $h\in Y$ (note that $_sh\in Y$ for $h\in Y$ from the assumption). This implies that $sT_r(k) = T_r(k)$ ($s\in S$, $k\in K$) since $Y$ separates points of $K$, and hence $T_s\circ T_r = T_r$ ($s\in S$). We thus have $T_s\circ T_t = T_{st}$ for $s\in S$ and $t\in S^r$. Moreover, Equation~(\ref{extension}) shows that $h_{T_r(k)}(s) = h(T_r(k))$ is a constant function on $S$. So
\[  h(T_r(T_r(k))) = m(h_{T_r(k)}) = h(T_r(k))m(1) = h(T_r(k))  \quad (k\in K, h\in Y).  \]
We then have $T_r\circ T_r = T_r$.
\end{proof}

The following theorem was obtained  in \cite[Theorem~3.8]{L-Z} for a separable semitopological semigroup $S$. We now remove the separability condition.

\begin{thm}\label{WAP MLIM}
Let $S$ be a semitopological semigroup. Then $WAP(S)$ has a multiplicative LIM if and only if the following fpp hold.
\begin{description}
\item[($F_{Eq}$)] Every separately continuous and quasi-equicontinuous representation of $S$ on a compact Hausdorff space $K$ has a common fixed point in $K$.
\end{description}
\end{thm}
\begin{proof}
Suppose that $WAP(S)$ has a multiplicative LIM. Let $\Ss=\{T_s:\, s\in S\}$ be a separately continuous and quasi-equicontinuous representation of $S$ on a compact Hausdorff space $K$. We consider $X = WAP(S)$ and $Y = C(K)$. Then $X$ and $Y$ satisfy the conditions in Lemma~\ref{MLIM}, where the condition $h_k \in X$ for $h\in C(K)$ and $k\in K$ holds due to \cite[Lemma~3.2]{L-Z}. So the representation $\Ss$ of $S$ on $K$ extends to $S^r$. Then every point in $T_r(K)$ is a common fixed point of $S$ in $K$.

For the converse we consider $E = WAP(S)^*$ with the topology determined by the family of seminorms $Q = \{p_f:\, f\in WAP(S)\}$, where $p_f (\phi) = \sup\{|\phi(\ell_sf)|, |\phi(f)|:\, s\in S\}$ ($\phi\in E$). Let $C$ be the set of means on $WAP(S)$. Then $C$ is a weakly compact convex subset of $(E,Q)$ (note that the weak topology of $(E,Q)$ coincides with the weak* topology $\sigma(WAP(S)^*, WAP(S))$ by the Machey-Arens theorem). It was shown in the counterpart of the proof of \cite[Theorem~3.4]{L-Z}: that  $\Ss = \{\ell_s^*: \,s\in S\}$ defines a weakly separately continuous and weakly quasi-equicontinuous representation on $C$. Let $K$ be the set of multiplicative means on $WAP(S)$. Then $K$ is a weakly compact $S$-invariant subset of $C$. So $\Ss$ is still separately continuous and quasi-equicontinuous as a representation of $S$ on $K$, where $K$ is endowed with the weak topology of $(E,Q)$, according to \cite[Lemma~3.1(2)]{L-Z}. By the hypothesis, $S$ has a common fixed point in $K$, which is indeed a multiplicative LIM on $WAP(S)$.
\end{proof}

It is not known to the authors whether the separability assumption in \cite[Theorem~3.4]{L-Z} is removable.

Recall that a function $f\in C_b(S)$ is \emph{left uniformly continuous} if the mapping $s\mapsto \ell_sf$ from $S$ into $C_b(S)$ is continuous. The set MM$(S)$ of all multiplicative means on $C_b(S)$ is a total subset of $C_b(S)^*$. It generates a locally convex topology $\tau_{MM} = \sigma(C_b(S), \text{MM}(S))$ on $C_b(S)$. A function $f\in C_b(S)$ is called \emph{left multiplicatively continuous} if the mapping $s\mapsto \ell_sf$ from $S$ into ($C_b(S),\tau_{MM})$ is continuous \cite{Mitch_LUC}.  Let $LUC(S)$ be the subspace of $C_b(S)$ consisting of all left uniformly continuous functions on $S$, and let $LMC(S)$ be the subspace of $C_b(S)$ consisting of all left multiplicatively continuous functions on $S$. Both $LUC(S)$ and $LMC(S)$ are left invariant closed subalgebras of $C_b(S)$ containing the constant functions. We have the following result which was first proved by T. Mitchell in \cite{Mitch_LUC}.

\begin{thm}\label{Thm Sr}
Let $S$ be a semitopological semigroup. Then
\begin{enumerate}
\item $LUC(S)$ has a multiplicative left invariant mean if and only if 
\begin{description}
\item[($F_j$)] every jointly continuous representation of $S$ on a nonempty compact Hausdorff space has a common fixed point for $S$;
\end{description}
\item $LMC(S)$ has a multiplicative left invariant mean if and only if
\begin{description}
\item[($F_{s}$)] every separately continuous representation of $S$ on a nonempty compact Hausdorff space has a common fixed point for $S$.
\end{description}
\end{enumerate}
\end{thm}

\begin{proof} By Lemma~\ref{S0} it suffices to consider whether the representation can be extended to a representation of $S^r$ on $K$ in each case.
For (a)
, to show the necessity we take $X = LUC(S)$ and $Y=C(K)$. 
Since the $S$ action on $K$ is jointly continuous and $K$ is compact, using standard $\varepsilon$-argument one sees that the mapping $s\mapsto {_sh}$ is a continuous map from $S$ into $C(K)$ for each $h\in C(K)$. This implies that, for each $k\in K$ and $h\in Y=C(K)$, $h_k \in LUC(S)=X$. So the conditions of Lemma~\ref{MLIM} hold. Thus, the representation of $S$ on $K$ may be extended to $S^r$.
To show the sufficiency, we consider the set $K$ of all multiplicative means on $LUC(S)$. With the weak* topology of $LUC(S)^*$, $K$ is a compact Hausdorff space. Let $T_s = \ell_s^*$ ($s\in S$), the dual operator of the left translation $\ell_s$ on $LUC(S)$. Then $\Ss = \{T_s:\; s\in S\}$ is a jointly continuous representation of $S$ on $K$. So it extends to a representation of $S^r$ on $K$ by the assumption. Clearly, every element of $T_r(K)$ is a multiplicative left invariant mean on $LUC(S)$.

For (b)
, to show the necessity we take $X = LMC(S)$ and $Y=C(K)$. If $\mu\in \text{MM}(S)$, then there is a net $(s_i)\subset S$ such that $\mu = \lim_i\delta_{s_i}$ in the weak* topology of $C_b(S)^*$. Given $k\in K$, passing to a subnet if necessary, we may assume $\lim_i s_ik = k'\in K$. Then $\langle \mu, \ell_s(h_k)\rangle = h(sk')$ is clearly continuous in $s\in S$ when $h\in Y$. This implies that $h_k\in LMC(S)$. From Lemma~\ref{MLIM} the representation of $S$ on $K$ may be extended to $S^r$ if $LMC(S)$ has a multiplicative left invariant mean. To show the sufficiency, we consider the set $K$ of all multiplicative means on $LMC(S)$. With the weak* topology of $LMC(S)^*$, $K$ is a compact Hausdorff space. Since each multiplicative mean on $LMC(S)$ may be extended to a multiplicative mean on $C_b(S)$, $\Ss = \{\ell_s^*:\; s\in S\}$ is a separately continuous representation of $S$ on $K$. So it extends to a representation of $S^r$ on $K$ by the assumption, showing that $LMC(S)$ has a multiplicative left invariant mean.
\end{proof}

Now we consider a convex subset $K$ of a separated locally convex topological vector space $E$. We denote by $A(K)$ the set of all continuous affine functions on $K$, that is, 
\begin{align*}  A(K) = \{f\in C(K): &\, f(\sum_{i=1}^n{\al_i x_i}) = \sum_{i=1}^n{\al_i f(x_i)},\\
                                 & \text{ for }\al_i>0, x_i\in K , \sum_{i=1}^n{\al_i} = 1\} .
                                 \end{align*} 
We note $E^*\subset A(K)$.

\begin{lemma}\label{convex}
Let $S$ be a semitopological semigroup acting as separately continuous self-mappings on a compact convex subset $K$ of a separated locally convex topological vector space. 
 Suppose that there is a left invariant subspace $X$ of $C_b(S)$ containing the constant functions and there is a subset $Y$ of $A(K)$ such that $Y$ separates points of $K$, $_sh \in Y$ and $h_k\in X$ for $h\in Y$, $s\in S$ and $k\in K$. If $X$ has a left invariant mean, then the representation of $S$ on $K$ can be extended to a representation of $S^r$ on $K$ and so $S$ has a common fixed point in $K$.
\end{lemma}
\begin{proof}
Suppose that $m$ is a left invariant mean on $X$. Then there is a net $(m_i)\in co\{\delta_s:\; s\in S\}$ such that $m=\lim_i m_i$ in the weak* topology of $X^*$. Let $m_i = \sum_j{\al^{(i)}_j\delta_{s^{(i)}_j}}$, where $s^{(i)}_j\in S$, $\al^{(i)}_j>0$ and $\sum_j{\al^{(i)}_j}=1$. Then
\[  m(h_k) = \lim_i \sum_j{\al^{(i)}_jh(s^{(i)}_jk)} = \lim_i h(\sum_j{\al^{(i)}_js^{(i)}_jk}) \in h(K)\]
for $h\in Y$ since $h$ is affine and $h(K)$ is compact. As in the proof of Lemma~\ref{MLIM} we denote
\[  K_{h,k} = \{x\in K :\, h(x) = m(h_k)\}, \quad (k\in K, h\in Y). \]
If $h^1, h^2, \cdots, h^n \in Y$ then the mapping $\Gamma$: $K\to \C^n$ defined by $\Gamma(x) = (h^1(x), h^2(x), \cdots, h^n(x))$ is continuous and affine, and $\Gamma(K)$ is compact. So,
\[ (m(h^1_k), m(h^2_k), \cdots, m(h^n_k)) = \lim_i\Gamma(\sum_j{\al^{(i)}_js^{(i)}_jk}) \in \Gamma(K) .  \]
This shows that $(K_{h,k})_{h\in Y}$ has finite intersection property. Thus, 
we have $\cap_{h\in Y} K_{h, k} \neq\emptyset$. Then the argument of the corresponding part of Lemma~\ref{MLIM} works here verbatim to establish the representation of $S^r$ on $K$.
 \end{proof}
 
 We recall that a function $f\in C_b(S)$ is a \emph{weakly left uniformly continuous} if the mapping $s\mapsto \ell_sf$ from $S$ into ($C_b(S),\text{wk})$ is continuous \cite{Mitch_LUC}. The set of all weakly left uniformly continuous functions on $S$ is a left invariant closed subspace of $C_b(S)$, denoted by $WLUC(S)$. Denote by M$(S)$ the set of all means 
on $C_b(S)$ and let $\tau_M = \sigma(C_b(S), \text{M}(S))$. Since $C_b(S)^*$ is linearly  spanned by $\text{M}(S)$, we have that $f\in WLUC(S)$ if and only if the mapping $s\mapsto \ell_sf$ from $S$ into ($C_b(S),\tau_M)$ is continuous. The following theorem is also due to T. Mitchell \cite{Mitch_LUC}.

\begin{thm}\label{Thm Sr convex}
Let $S$ be a semitopological semigroup. Then
\begin{enumerate}
\item $LUC(S)$ has a left invariant mean if and only if
\begin{description}\label{affine Fc}
\item[($F_{ja}$)] every jointly continuous affine representation of $S$ on a nonempty convex compact subset of a separated locally convex topological vector space has a common fixed point for $S$;
\end{description}
\item $WLUC(S)$ has a left invariant mean if and only if
\begin{description}
\item[($F_{sa}$)] every separately continuous affine representation of $S$ on a compact convex subset of a separated locally convex topological vector space has a common fixed point for $S$.
\end{description}
\end{enumerate}
\end{thm} 
 \begin{proof}
 We use Lemma~\ref{convex}. The proof is similar to that of Theorem~\ref{Thm Sr}. For (a)
 , we take $X = LUC(S)$, $Y=A(K)$; and for (b) 
   we take $X = WLUC(S)$, $Y=A(K)$. We note $_sh\in Y$ if $h\in Y$ and $s\in S$ since the $S$ representation is affine. 
\end{proof}



\begin{remark}
We note that Theorem~\ref{Thm Sr convex}.(a) 
 is no longer true if one replaces ``affine'' with ``convex''. Here by a convex map (or a convex function) we mean a map (resp. a function) defined on a convex set $K$ that maps each convex subset of $K$ onto a convex set. For example any continuous function on an interval is a convex function. A counter example is as follows: Let $f$ and $g$ be two continuous functions on the unit interval $[0,1]$ such that $f$ commutes with $g$ under composite product but they do not have a common fixed point in $[0,1]$. We consider the free discrete commutative semigroup $SC_2$ on two generators $s_1$ and $s_2$. Then $SC_2$ is amenable. Define  on $[0,1]$ by $T_{s_1}(x)=f(x)$ and $T_{s_2}(x)=g(x)$ ($x\in [0,1]$). They generate a continuous convex representation of $SC_2$ on $[0,1]$.  This representation does not have a common fixed point in $[0,1]$.
\end{remark}

\section{Some remarks and open problems}\label{open questions}

The following is a long-standing open problem.
\begin{question}\label{Q F*}
Does a semitopological semigroup $S$ has the fpp ($F_*$) if $LUC(S)$ has a LIM?
\end{question}
The question is open even for discrete case \cite{Lau.survey}. Here we present a theorem which asserts that a weak version of the fixed point property ($F_*$) holds if $LUC(S)$ has a LIM.
 
\begin{prop}
Let $S$ be a semitopological semigroup. If $LUC(S)$ has a LIM, then $S$ has the following fixed point property.
\begin{description}
\item[($F_{*n}$)]\label{F*n} Let $\Ss = \{T_s:\; s\in S\}$ be a norm nonexpansive representation of $S$ on a nonempty weak* compact convex set $C$ of the dual space of a Banach space $E$. If $C$ has norm normal structure and the mapping $(s,x)\mapsto T_s(x)$ from $S\times C$ to $C$ is jointly continuous when $C$ is endowed with the weak* topology of $E^*$, then there is a common fixed point for $S$ in $C$.
\end{description}
\end{prop}

\begin{proof}
Let $K$ be a minimal weak* closed convex $S$-invariant subset of $C$ and $F$ be a nonempty minimal weak* closed $S$-invariant subset of $K$. Take a $y\in F$. Since the action of $S$ on $C$ is jointly weak* continuous, for each $f\in C(F)$ the mapping $s\mapsto f(sy)$ belongs to $LUC(S)$, where $F$ is equipped with the weak* topology of $E^*$. Using a left invariant mean on $LUC(S)$ and the standard argument of \cite[Theorem~4.1]{Lau73}, one sees that $T_sF=F$ holds for each $s\in S$. Then the norm normal structure and the norm nonexpansiveness implies that $F$ must be a singleton. So the single point in $F$ is a common fixed point for $S$.
\end{proof}

Using our Lemma~\ref{T(B)=B} combined with Lemma~5.1 and 5.2 of \cite{L-T}, we have the following result related to Question~\ref{Q F*}, which extends \cite[Theorem~5.3]{L-T}.

\begin{thm}
Let $S$ be a semitopological semigroup. If $S$ is left reversible or left amenable then the following fixed point property holds.
\begin{description}
\item[($F_{*s}$)]  Whenever $\Ss = \{T_s:\; s\in S\}$ is a norm nonexpansive representation of $S$ on a nonempty norm separable weak* compact convex set $C$ of the dual space of a Banach space $E$ and the mapping $(s,x)\mapsto T_s(x)$ from $S\times C$ to $C$ is jointly continuous when $C$ is endowed with the weak* topology of $E^*$, then there is a common fixed point for $S$ in $C$.
\end{description}
\end{thm}
\qed

\begin{question}
Let $S$ be a (discrete) semigroup. If the fpp ($F_{*s}$) holds, does $WAP(S)$ have a LIM? We also do not know whether the existence of a LIM on $WAP(S)$ is sufficient to ensure the fpp ($F_{*s}$).
\end{question}
Here we give some partial affirmative results. We first note that $S$ has the fpp ($F_{*s}$) if and only if its unitization $S^e$ does. So any consequence of ($F_{*s}$) for $S^e$ is also a consequence of ($F_{*s}$) for $S$. The next lemma may already be well-known. But we still include a proof for completeness.

\begin{lemma}\label{LIM Se to S}
For a semigroup $S$, $AP(S^e)$ has a LIM if and only if $AP(S)$ does. Similarly, $WAP(S^e)$ has a LIM if and only if $WAP(S)$ does.
\end{lemma}
\begin{proof}
For $f\in \ell^\infty(S)$ we define $f^e\in \ell^\infty(S^e)$ by $f^e(s)=f(s)$ if $s\in S$ and $f^e(e)=0$. We have, for each $a\in S$,
\[  \ell_af^e(s) =\begin{cases} \ell_af(s) & \text{ if $s\in S$,}\\
                                       f(a)  &\text{ if $s=e$.}
\end{cases}  \]
Therefore, $f^e\in AP(S^e)$ if $f\in AP(S)$; and $f^e\in WAP(S^e)$ if $f\in WAP(S)$. Further, we have $(\ell_af)^e = \ell_af^e - f(a)\delta_e$ for each $a\in S$. Now, if $\mathbf m$ is a left invariant mean on $AP(S^e)$, we let $\tilde{\mathbf m} \in AP(S)^*$ be defined by $\langle f, \tilde{\mathbf m}\rangle = \langle f^e, \mathbf m\rangle$ ($f\in AP(S)$). Clearly, $\tilde{\mathbf m}$ is a mean on $AP(S)$. Moreover, since $\ell_a\delta_e = 0$ for $a\in S$, we have $\langle \delta_e, \mathbf m\rangle =\langle \ell_a\delta_e, \mathbf m\rangle = 0$ and hence 
\begin{align*} \langle \ell_af, \tilde{\mathbf m}\rangle &= \langle (\ell_af)^e, \mathbf m\rangle =\langle \ell_af^e, \mathbf m\rangle - f(a)\langle \delta_e, \mathbf m\rangle\\
 &=\langle \ell_af^e, \mathbf m\rangle = \langle f^e, \mathbf m\rangle = \langle f, \Tilde{\mathbf m}\rangle. 
\end{align*}
So $\tilde{\mathbf m}$ is a LIM on $AP(S)$. Similarly, if $\mathbf m$ is a LIM on $WAP(S^e)$, then $\tilde{\mathbf m}$ is a LIM on $WAP(S)$.

For the converse, we denote $\tilde f = f|_S$ if $f\in \ell^\infty (S^e)$. Then $\tilde f \in AP(S)$ (resp. $WAP(S)$) if $f\in AP(S^e)$ (resp. $WAP(S^e)$), and $\ell_a\tilde f = (\ell_a f)\tilde{} $ for each $a\in S$. If $\mu\in AP(S)^*$ is a LIM on $AP(S)$, then let $\mathbf m \in AP(S^e)^*$ be defined by $\langle f, \mathbf m\rangle = \langle \tilde f, \mu\rangle$ ($f\in AP(S^e)$). Then $\mathbf m$ is a mean on $AP(S^e)$ and
\[
  \langle \ell_af, \mathbf m\rangle = \langle (\ell_af)\tilde{}, \mu\rangle =\langle \ell_a\tilde f, \mu\rangle = \langle \tilde f, \mu\rangle = \langle f, \mathbf m\rangle \quad (a\in S^e).  \]
Hence $\mathbf m$ is a LIM on $AP(S^e)$. Similarly, if $\mu$ is a LIM on $WAP(S)$, then $\mathbf m$ is a LIM on $WAP(S^e)$.
\end{proof}

\begin{remark}\label{orbit}
Note that when $X=AP(S)$ or $WAP(S)$ or $LUC(S)$, then $X$ is introverted, that is if $m\in X^*$ and $f\in X$ then the function $h(s) = m(\ell_s f)$ ($s\in S$) is also in $X$. In this case $X$ has a LIM if and only if $\overline{co}^p\{r_s f: s\in S\}$, the pointwise closure of the right orbit of $f$, contains a constant function for every $f\in X$ (see \cite[Theorem~3.1]{Lau-Proc} or \cite{Gran-Lau}).
\end{remark}

\begin{prop}\label{FsLIM} Suppose that $S$ has the fixed point property $(F_{*s})$. Then
\begin{enumerate}
\item $AP(S)$ has a LIM;\label{AP}
\item $WAP(S)$ has a LIM if $S$ has a countable left ideal.\label{WAP}
\end{enumerate}
\end{prop}
\begin{proof}
From Lemma~\ref{LIM Se to S} and the remark before it, we may assume $S$ is unital.

For (a): Let $f\in AP(S)$ and let $K_f = \overline{co}(r_sf:\;s\in S)$ be the norm closure of the convex hull of the right orbit $\mathcal{RO}f$ of $f$. Then $K_f$ is a compact and hence weak* compact convex subset of $\ell^\infty(S)= \ell^1(S)^*$. As a norm compact set $K_f$ is norm separable. Take the natural representation of $S$ on $K_f$ defined by the right translation mapping $T_sg = r_sg$ ($g\in K_f$). Then $T_s$ is norm nonexpansive and weak* continuous in $s\in S$ (note that the weak* topology coincides with the norm topology on $K_f$. By ($F_{*s}$), $K_f$ has a fixed point $g\in K_f$ such that $r_sg =g$ ($s\in S$). This $g$ must be a constant function. Therefore, each $K_f$ ($f\in AP(S)$) contains a constant function. Thus $AP(S)$ has a LIM by Remark~\ref{orbit}.

For (b): Let $S_0$ be a countable left ideal of $S$. Then for each $f\in WAP(S)$, $K_f^0 = \overline{co}(r_sf:\; s\in S_0)$ is a weakly compact (so also weak* compact), norm separable convex subset of $\ell^\infty(S)$ which is invariant under right translations. Still we consider the right translation representation on $K_f^0$. The representation is weakly continuous and hence is weak* continuous on $K_f^0$ since the two topologies coincides on $K_f^0$. From ($F_{*s}$), there is a fixed point for this representation in $K_f^0$ which must be a constant. So $K_f \supset K_f^0$ contains a constant function for each $f\in WAP(S)$. Consequently, $WAP(S)$ has a left invariant mean as in (a).

\end{proof}

\begin{question}
Does the partially bicyclic semigroup $S_2 = \la e, a, b, c: ab=ac=e\ra$ has the fpp $(F_{*s})$?
\end{question}

 It is known that $S_2
 $, 
is not left reversible, but $WAP(S_2)$ has a left invariant mean \cite{L-Z}. If the answer to the above question is yes, then $S$ having $(F_{*s})$ is not equivalent to $S$ being left reversible; if the answer is no, then the converse of Proposition~\ref{FsLIM}~(b) does not hold even for countable semigroup $S$.

Our Theorem~\ref{WAP MLIM} improves \cite[Theorem~3.8]{L-Z} by removing the separability condition on $S$ assumed there. We wonder if the same thing can be done to \cite[Theorem~3.4]{L-Z}. Precisely, we raise
\begin{question}
Does the following fpp hold if $S$ is a semitopological semigroup and $WAP(S)$ has a LIM? 
\begin{description}
\item[($F$)] every weakly separately continuous, weakly quasi-equicontinuous and $Q$-nonexpansive representation of $S$ on a weakly compact convex subset of a separated locally convex topological vector space $(E,Q)$ has a common fixed point for $S$.
\end{description}
\end{question} 

\begin{question}
What amenability property of a topological group or semigroup may be characterized by the Schauder fixed point property?
\end{question}

\begin{acknowledgement}
The authors are grateful to the referee for his or her very careful reading of the paper, valuable suggestions, and references \cite{B-L} and \cite{D-D-S-T}.
\end{acknowledgement}

\end{document}